\documentclass[11pt,oneside,english,reqno]{amsart}
\usepackage[LGR,T1]{fontenc}
\usepackage[latin9]{inputenc}
\usepackage{geometry}
\usepackage{babel}
\usepackage{float}
\usepackage{mathrsfs}
\usepackage{amsthm}
\usepackage{amstext}
\usepackage{amssymb}
\usepackage{graphicx}
\usepackage{esint}
\usepackage[all]{xy}
\PassOptionsToPackage{normalem}{ulem}
\usepackage{ulem}
\usepackage[unicode=true,pdfusetitle,
 bookmarks=true,bookmarksnumbered=false,bookmarksopen=false,
 breaklinks=false,pdfborder={0 0 1},backref=false,colorlinks=false]
 {hyperref}
\hypersetup{
 colorlinks=true,citecolor=blue,linkcolor=blue,linktocpage=true}

\makeatletter

\DeclareRobustCommand{\greektext}{%
  \fontencoding{LGR}\selectfont\def\encodingdefault{LGR}}
\DeclareRobustCommand{\textgreek}[1]{\leavevmode{\greektext #1}}
\DeclareFontEncoding{LGR}{}{}
\DeclareTextSymbol{\~}{LGR}{126}
\providecommand{\tabularnewline}{\\}

\numberwithin{equation}{section}
\numberwithin{figure}{section}
\numberwithin{table}{section}
\usepackage{enumitem}		
\theoremstyle{plain}
\newtheorem{thm}{\protect\theoremname}[section]
  \theoremstyle{remark}
  \newtheorem{rem}[thm]{\protect\remarkname}
  \theoremstyle{definition}
  \newtheorem{defn}[thm]{\protect\definitionname}
  \theoremstyle{plain}
  \newtheorem{lem}[thm]{\protect\lemmaname}
  \theoremstyle{plain}
  \newtheorem{cor}[thm]{\protect\corollaryname}
  \theoremstyle{definition}
  \newtheorem{example}[thm]{\protect\examplename}
  \theoremstyle{remark}
  \newtheorem*{acknowledgement*}{\protect\acknowledgementname}

\allowdisplaybreaks

\providecommand{\MR}[1]{}

\usepackage{mathtools}
\usepackage{refstyle}
\newref{thm}{
         name      = {theorem~},
         names     = {theorems~},
         Name      = {Theorem~},
         Names     = {Theorems~},
         rngtxt    = \RSrngtxt,
         lsttwotxt = \RSlsttxt,
         lsttxt    = \RSlsttxt
         }  
\newref{prop}{
         name      = {proposition~},
         names     = {propositions~},
         Name      = {Proposition~},
         Names     = {Propositions~},
         rngtxt    = \RSrngtxt,
         lsttwotxt = \RSlsttxt,
         lsttxt    = \RSlsttxt
         }   
         
\newref{cor}{
         name      = {corollary~},
         names     = {corollaries~},
         Name      = {Corollary~},
         Names     = {Corollaries~},
         rngtxt    = \RSrngtxt,
         lsttwotxt = \RSlsttxt,
         lsttxt    = \RSlsttxt
         }   
         \newref{rem}{
         name      = {remark~},
         names     = {remarks~},
         Name      = {Remark~},
         Names     = {Remarks~},
         rngtxt    = \RSrngtxt,
         lsttwotxt = \RSlsttxt,
         lsttxt    = \RSlsttxt
         }   
         \newref{ex}{
         name      = {example~},
         names     = {examples~},
         Name      = {Example~},
         Names     = {Examples~},
         rngtxt    = \RSrngtxt,
         lsttwotxt = \RSlsttxt,
         lsttxt    = \RSlsttxt
         }   
         
         \newref{def}{
         name      = {definition~},
         names     = {definitions~},
         Name      = {Definition~},
         Names     = {Definitions~},
         rngtxt    = \RSrngtxt,
         lsttwotxt = \RSlsttxt,
         lsttxt    = \RSlsttxt
         }

\renewcommand{\section}{%
\@startsection{section}{1}%
  \z@{.7\linespacing\@plus\linespacing}{.5\linespacing}%
  {\normalfont\scshape\centering\bfseries}}

\renewcommand{\subsection}{%
\@startsection{subsection}{2}%
  \z@{.5\linespacing\@plus.7\linespacing}{.5\linespacing}%
  {\normalfont\bfseries}}

\renewcommand{\subsubsection}{%
\@startsection{subsubsection}{2}%
  \z@{.5\linespacing\@plus.7\linespacing}{.5\linespacing}%
  {\normalfont\bfseries}}

\newref{sec}{%
      name      = \RSsectxt,
      names     = \RSsecstxt,
      Name      = \RSSectxt,
      Names     = \RSSecstxt,
      refcmd    = {\ifRSstar\S\fi\ref{#1}},
      rngtxt    = \RSrngtxt,
      lsttwotxt = \RSlsttwotxt,
      lsttxt    = \RSlsttxt}

\theoremstyle{definition}

\AtBeginDocument{
  
}

\makeatother

  \providecommand{\acknowledgementname}{Acknowledgement}
  \providecommand{\corollaryname}{Corollary}
  \providecommand{\definitionname}{Definition}
  \providecommand{\examplename}{Example}
  \providecommand{\lemmaname}{Lemma}
  \providecommand{\remarkname}{Remark}
\providecommand{\theoremname}{Theorem}

\begin{document}

\title{Infinite networks and variation of conductance functions in discrete
Laplacians}

\author{Palle Jorgensen and Feng Tian}

\address{(Palle E.T. Jorgensen) Department of Mathematics, The University
of Iowa, Iowa City, IA 52242-1419, U.S.A. }

\email{palle-jorgensen@uiowa.edu}

\urladdr{http://www.math.uiowa.edu/\textasciitilde{}jorgen/}

\address{(Feng Tian) Department of Mathematics, Wright State University, Dayton,
OH 45435, U.S.A.}

\email{feng.tian@wright.edu}

\urladdr{http://www.wright.edu/\textasciitilde{}feng.tian/}



\maketitle
\pagestyle{myheadings}
\markright{Variation of Conductance Functions in Discrete Laplacians}
\begin{abstract}
For a given infinite connected graph $G=(V,E)$ and an arbitrary but
fixed conductance function $c$, we study an associated graph Laplacian
$\Delta_{c}$; it is a generalized difference operator where the differences
are measured across the edges $E$ in $G$; and the conductance function
$c$ represents the corresponding coefficients. The graph Laplacian
(a key tool in the study of infinite networks) acts in an energy Hilbert
space $\mathscr{H}_{E}$ computed from $c$. Using a certain Parseval
frame, we study the spectral theoretic properties of graph Laplacians.
In fact, for fixed $c$, there are two versions of the graph Laplacian,
one defined naturally in the $l^{2}$ space of $V$, and the other
in $\mathscr{H}_{E}$. The first is automatically selfadjoint, but
the second involves a Krein extension. We prove that, as sets, the
two spectra are the same, aside from the point 0. The point zero may
be in the spectrum of the second, but not the first.

We further study the fine structure of the respective spectra as the
conductance function varies; showing now how the spectrum changes
subject to variations in the function $c$. Specifically, we study
an order on the spectra of the family of operators $\Delta_{c}$,
and we compare it to the ordering of pairs of conductance functions.
We show how point-wise estimates for two conductance functions translate
into spectral comparisons for the two corresponding graph Laplacians;
involving a certain similarity: We prove that point-wise ordering
of two conductance functions $c$ on $E$, induces a certain similarity
of the corresponding (Krein extensions computed from the) two graph
Laplacians $\Delta_{c}$. The spectra are typically continuous, and
precise notions of fine-structure of spectrum must be defined in terms
of equivalence classes of positive Borel measures (on the real line.)
Our detailed comparison of spectra is analyzed this way.
\end{abstract}

\tableofcontents{}

\section{Introduction}

By an electrical network we mean a graph $G$ of vertices and edges
satisfying suitable conditions which allow for computation of voltage
distribution from a network of prescribed resistors assigned to the
edges in $G$. The mathematical axioms are prescribed in a way that
facilitates the use of the laws of Kirchhoff and Ohm in computing
voltage distributions and resistance distances in $G$. It will be
more convenient to work with prescribed conductance functions $c$
on $G$. Indeed with a choice of conductance function $c$ specified
we define two crucial tools for our analysis, a graph Laplacian $\Delta\left(=\Delta_{c},\right)$
a discrete version of more classical notions of Laplacians, and an
energy Hilbert space $\mathscr{H}_{E}$.

To make it more clear what are the new contributions in the present
paper relative to previous work (e.g., \cite{JoPe10,JoPe11a,JoPe11b,JT14}),
we note that they are two-fold: our spectral theoretic conclusions
in Section \ref{sec:sp} below, and our applications in Section \ref{sec:vc}.
The key to our construction in sect \ref{sec:sp} is our identification
of a canonical closable operator $L_{0}$ (with dense domain) from
$l^{2}$ into $\mathscr{H}_{E}$. Once closability is established,
from the graph-closure $L$, we then get the following two selfadjoint
operators $LL^{*}$ in $\mathscr{H}_{E}$, and $L^{*}L$ in $l^{2}$.
We show that the first is a Krein extension (referring to $\mathscr{H}_{E}$)
, and that the second is the selfadjoint $l^{2}$-Laplacian $\Delta_{2}$.
Of course $l^{2}$ does not depend on choice of conductance function,
but the graph Laplacian $\Delta$ does. A number of issues are resolved,
and care is exercised: Since $L_{0}$ maps between different Hilbert
spaces, i.e., from a dense domain in $l^{2}$ to $\mathscr{H}_{E}$,
its adjoint $L_{0}^{*}$ acts in the reverse direction. We show that
the domain of $L_{0}^{*}$ is dense in $\mathscr{H}_{E}$, and so
$L_{0}$ is closable. From this we proceed to establish unitary equivalence
between an associated pair of selfadjoint operators. Specifically
we show that two selfadjoint operators $LL^{*}$ in $\mathscr{H}_{E}$,
and $L^{*}L$ are unitarily equivalent; i.e., that the corresponding
spectral measures are unitarily equivalent, and we specify an explicit
intertwining operator.

Note that unitary equivalence is a global conclusion, and that it
is a much stronger conclusion than related local notions of ``same
spectrum,'' such as comparison of local Radon-Nikodym derivatives
for the respective spectral measures computed in cyclic subspaces.

Because of statistical consideration, and our use of random walk models,
we focus our study on infinite electrical networks, i.e., the case
when the graph $G$ is countable infinite. In this case, for realistic
models the graph Laplacian $\Delta_{c}$ will then be an unbounded
operator with dense domain in $\mathscr{H}_{E}$, Hermitian and semibounded.
Hence it has a unique Krein extension. Our main theorem gives an answer
to how the Krein extensions depend on assignment of conductance function.
Our theorem offers a direct comparison the Krein extensions $\Delta_{c}$
associated to pairs conductance functions of say $c_{1}$, and $c_{2}$
both defined on the same $G$.

Large networks arise in both pure and applied mathematics, e.g., in
graph theory (the mathematical theory of networks), and more recently,
they have become a current and fast developing research area; with
applications including a host of problems coming from for example
internet search, and social networks. Hence, of the recent applications,
there is a change in outlook from finite to infinite.

More precisely, in traditional graph theoretical problems, the whole
graph is given exactly, and we are then looking for relationships
between its parameters, variables and functions; or for efficient
algorithms for computing them. By contrast, for very large networks
(like the Internet), variables are typically not known completely;
-- in most cases they may not even be well defined. In such applications,
data about them can only be collected by indirect means; hence random
variables and local sampling must be used as opposed to global processes.

Although such modern applications go far beyond the setting of large
electrical networks (even the case of infinite sets of vertices and
edges), it is nonetheless true that the framework of large electrical
networks is helpful as a basis for the analysis we develop below;
and so our results will be phrased in the setting of large electrical
networks, even though the framework is much more general.

The applications of \textquotedblleft large\textquotedblright{} or
infinite graphs are extensive, counting just physics; see for example
\cite{BCD06,RAKK05,KMRS05,BC05,TD03,VZ92}.

\section{Preliminaries}

Starting with a given network $(V,E,c)$, we introduce functions on
the vertices $V$, voltage, dipoles, and point-masses; and on the
edges $E$, conductance, and current. We introduce the graph-Laplacian
$\Delta$ (see Definition \ref{def:lap} below.) There are two Hilbert
spaces serving different purposes, $l^{2}(V)$, and the energy Hilbert
space $\mathscr{H}_{E}$; the latter depending on choice of conductance
function $c$.

The graph Laplacian $\Delta$ (Definition \ref{def:lap}) has an easy
representation as a densely defined semibounded operator in $l^{2}\left(V\right)$
via its matrix representation, see Remark \ref{rem:lapl2}. To do
this we use implicitly the standard orthonormal (ONB) basis $\left\{ \delta_{x}\right\} $
in $l^{2}(V)$. But in network problems, and in metric geometry, $l^{2}(V)$
is not useful; rather we need the energy Hilbert space $\mathscr{H}_{E}$
(see Section \ref{sub:EnergyH}.)

We are motivated in part by earlier papers on analysis on discrete
networks, see e.g., \cite{Ana11,Gri10,Zem96,Bar93,Tet91}.

Properties of $\mathscr{H}_{E}$: The case when $V$ is countably
infinite, i.e., $\#V=\aleph_{0}$.
\begin{enumerate}[label=(\roman{enumi}),ref=\roman{enumi}]
\item The equation $\Delta h=0$ typically has non-zero solutions in $\mathscr{H}_{E}$,
but not in $l^{2}\left(V\right)$. 
\item Let $x$ and $y$ be two distinct vertices; then the equation 
\begin{equation}
\Delta v_{xy}=\delta_{x}-\delta_{y}\label{eq:E1}
\end{equation}
has a solution in $\mathscr{H}_{E}$, but $v_{xy}\notin l^{2}\left(V\right)$. 
\item \label{enu:Ed}Given $x\neq y$ in $V$, then there is a unique $v_{xy}\in\mathscr{H}_{E}$
(dipole) such
\begin{equation}
\left\langle v_{xy},u\right\rangle _{\mathscr{H}_{E}}=u\left(x\right)-u\left(y\right),\label{eq:E2}
\end{equation}
and $v_{xy}$ satisfies (\ref{eq:E1}); i.e., (\ref{eq:E2}) implies
(\ref{eq:E1}).
\item Fix a base-point $o\in V$, and for $x\in V':=V\backslash\{o\}$,
set $v_{x}:=v_{xo}$; and 
\begin{equation}
G\left(x,y\right):=\left\langle v_{x},v_{y}\right\rangle _{\mathscr{H}_{E}},\label{eq:E3}
\end{equation}
then 
\begin{equation}
\Delta_{y}G\left(x,y\right)=\delta_{x,y}.\label{eq:E4}
\end{equation}
The function $G$ in (\ref{eq:E3}) is called the Gramian. 
\item The dipole vectors $v_{xy}$ from (\ref{enu:Ed}) satisfy the following:
\[
\left\Vert v_{xy}\right\Vert _{\mathscr{H}_{E}}^{2}=\sup\left\{ \frac{1}{\left\Vert u\right\Vert _{\mathscr{H}_{E}}^{2}}\:\Big|\: u\in\mathscr{H}_{E},u\left(x\right)=1,u\left(y\right)=0\right\} 
\]
and 
\begin{equation}
dist_{c}\left(x,y\right)=\left\Vert v_{xy}\right\Vert _{\mathscr{H}_{E}}^{2}=v_{xy}\left(x\right)-v_{xy}\left(y\right)\label{eq:E5}
\end{equation}
is a metric on $V$.\end{enumerate}
\begin{rem}
In the notation of electrical networks, (\ref{eq:E5}) states that
the distance between $x$ and $y$ is the voltage drop measured in
the dipole. 
\end{rem}

\begin{rem}[Properties of the metric $d_{c}$ in (\ref{eq:E5})]
Combining (\ref{eq:E5}) and (\ref{eq:E2}) we immediately obtain
the following estimate
\[
\left|u\left(x\right)-u\left(y\right)\right|^{2}\leq dist_{c}\left(x,y\right)\left\Vert u\right\Vert _{\mathscr{H}_{E}}^{2}
\]
valid for all $u\in\mathscr{H}_{E}$ and all pairs of vertices $x$
and $y$ in $V$. 

\uline{Consequence (i):} If $\widetilde{V}_{c}$ denotes the completion
of the metric space $\left(V,dist_{c}\right)$, then every $u\in\mathscr{H}_{E}$
extends by completion to be a continuous Lipschitz-function (denoted
$\widetilde{u}$) on $\widetilde{V}_{c}$, and we have 
\[
\left|\widetilde{u}\left(x^{*}\right)-\widetilde{u}\left(y^{*}\right)\right|^{2}\leq\widetilde{dist_{c}}\left(x^{*},y^{*}\right)\left\Vert u\right\Vert _{\mathscr{H}_{E}}^{2}
\]
valid for all points $x^{*},y^{*}$ in $\widetilde{V}_{c}$.

\uline{Consequence (ii):} With notations as above; we conclude
that the space of continuous functions 
\[
\left\{ \widetilde{u}\;\mbox{on }\widetilde{V}_{c}\:\big|\:\left\Vert u\right\Vert _{\mathscr{H}_{E}}\leq1\right\} 
\]
is relatively compact in $C(\widetilde{V}_{c})$. \end{rem}
\begin{proof}[Proof of (ii)]
 Follows from (i), and the Arzelà-Ascoli theorem. 
\end{proof}

\subsection{\label{sub:setting}Basic Setting}

Here we define the graph Laplacian $\Delta\left(=\Delta_{c}\right)$,
and the energy Hilbert space $\mathscr{H}_{E}$, and we outline some
of their properties.

Let $V$ be a countable discrete set, and let $E\subset V\times V$
be a subset such that:
\begin{enumerate}
\item $\left(xy\right)\in E\Longleftrightarrow\left(yx\right)\in E$; $x,y\in V$;
\item $\#\left\{ y\in V\:|\:\left(xy\right)\in E\right\} $ is finite, and
$>0$ for all $x\in V$; 
\item $\left(xx\right)\notin E$; and
\item \label{enu:a4}$\exists\, o\in V$ s.t. for all $y\in V$ $\exists\, x_{0},x_{1},\ldots,x_{n}\in V$
with $x_{0}=o$, $x_{n}=y$, $\left(x_{i-1}x_{i}\right)\in E$, $\forall i=1,\ldots,n$.
(This property is called connectedness.)
\item If a conductance function $c$ is given we require $c_{x_{i-1}x_{i}}>0$.
See Definition \ref{def:cond} below.\end{enumerate}
\begin{defn}
\label{def:cond}A function $c:E\rightarrow\mathbb{R}_{+}$ is called
\emph{conductance function} if $c_{xy}>0$, $c_{xy}=c_{yx}$, for
all $\left(xy\right)\in E$. Given $x\in V$, we set 
\begin{equation}
c\left(x\right):=\sum_{\left(xy\right)\in E}c_{xy}.\label{eq:cond}
\end{equation}
The summation in (\ref{eq:cond}) is denoted $x\sim y$; i.e., $x\sim y$
if $\left(xy\right)\in E$. 
\end{defn}

\begin{defn}
\label{def:lap}When $c$ is a conductance function (see Def. \ref{def:cond})
we set $\Delta=\Delta_{c}$ (the corresponding graph Laplacian)
\begin{equation}
\left(\Delta u\right)\left(x\right)=\sum_{y\sim x}c_{xy}\left(u\left(x\right)-u\left(y\right)\right)=c\left(x\right)u\left(x\right)-\sum_{y\sim x}c_{xy}u\left(y\right).\label{eq:lap}
\end{equation}
\end{defn}
\begin{rem}
\label{rem:lapl2}Given $G=\left(V,E,c\right)$ as above, and let
$\Delta=\Delta_{c}$ be the corresponding graph Laplacian. With a
suitable ordering on $V$, we obtain the following banded $\infty\times\infty$
matrix-representation for $\Delta$ (eq. (\ref{eq:lapm})). We refer
to \cite{GLS12} for a number of applications of infinite banded matrices.
\begin{equation}
\begin{bmatrix}c\left(x_{1}\right) & -c_{x_{1}x_{2}} & 0 & \cdots & \cdots & \cdots & \cdots & 0 & \cdots\\
-c_{x_{2}x_{1}} & c\left(x_{2}\right) & -c_{x_{2}x_{3}} & 0 & \cdots & \cdots & \cdots & \vdots & \cdots\\
0 & -c_{x_{3}x_{2}} & c\left(x_{3}\right) & -c_{x_{3}x_{4}} & 0 & \cdots & \cdots & \huge\mbox{0} & \cdots\\
\vdots & 0 & \ddots & \ddots & \ddots & \ddots & \vdots & \vdots & \cdots\\
\vdots & \vdots & \ddots & \ddots & \ddots & \ddots & 0 & \vdots & \cdots\\
\vdots & \huge\mbox{0} & \cdots & 0 & -c_{x_{n}x_{n-1}} & c\left(x_{n}\right) & -c_{x_{n}x_{n+1}} & 0 & \cdots\\
\vdots & \vdots & \cdots & \cdots & 0 & \ddots & \ddots & \ddots & \ddots
\end{bmatrix}\label{eq:lapm}
\end{equation}

\end{rem}

(The above planar matrix representation for $\Delta$ in eq (\ref{eq:lap})
is an oversimplification in two ways: First, in the general case,
the width of the band down the diagonal in (\ref{eq:lapm}) is typically
more than 3; i.e., the matrix representation for the general graph
Laplacian is typically banded with much wider bands. Secondly, the
width is not in general constant, i.e., for each row in the infinite
by infinite matrix, the number of non-zero entries may vary and even
be unbounded. Furthermore the row and column size in the matrix is
typically double infinite, and if a base point is chosen in $V$,
it may occur inside the matrix, in a diagonal position.)

\vspace{1em}

\textbf{Overview of main results in the paper.} Our main results from
sections \ref{sec:sp} and \ref{sec:vc} are as follows: In section
\ref{sec:sp}, we prove the following result (Theorem \ref{thm:LLs}):
Starting with a fixed graph $(V,E)$ and an arbitrary but fixed conductance
function $c$, as noted, we then arrive at two versions of a selfadjoint
graph Laplacian, one defined naturally in the $l^{2}$ space of $V$,
and the other in the energy Hilbert space $\mathscr{H}_{E}$ defined
from $c$. The first is automatically selfadjoint, but the second
involves a Krein extension (see Definition \ref{def:kr}). We prove
that, as sets, the two spectra are the same, aside from the point
$0$. The point zero may be in the spectrum of the second, but not
the first. In addition to this theorem, we isolate other spectral
similarities.

In section \ref{sec:vc} we turn to fine structure of the respective
spectra. We study how the spectrum changes subject to change of conductance
function $c$. Specifically, we study how the spectrum of the graph
Laplacian $\Delta_{c}$ changes subject to variations in choice of
conductance function c, how estimates for two conductance functions
translate into spectra comparisons for the two corresponding graph
Laplacians. We prove (Theorems \ref{thm:cFri} and \ref{thm:Deltasp})
that the natural order of conductance functions, i.e., pointwise as
functions on $E$, induces a certain similarity of the corresponding
(Krein extensions computed from the) two graph Laplacians. Since the
spectra are typically continuous, precise notions of fine-structure
of spectrum must be defined in terms of equivalence classes of positive
Borel measures (on the real line.) Hence our detailed comparison of
spectra must be phrased involving these; see Definition \ref{def:spmeas}.

\subsection{\label{sub:EnergyH}The Energy Hilbert Spaces $\mathscr{H}_{E}$ }

Here we prove some technical lemmas for the energy Hilbert space $\mathscr{H}_{E}$,
and operators which will be needed later.

Let $G=\left(V,E,c\right)$ be an infinite connected network introduced
in section \ref{sub:setting}. Set $\mathscr{H}_{E}:=$ completion
of all compactly supported functions $u:V\rightarrow\mathbb{C}$ with
respect to 
\begin{align}
\left\langle u,v\right\rangle _{\mathscr{H}_{E}} & :=\frac{1}{2}\underset{\left(x,y\right)\in E}{\sum\sum}c_{xy}(\overline{u\left(x\right)}-\overline{u\left(y\right)})\left(v\left(x\right)-v\left(y\right)\right)\label{eq:Einn}\\
\left\Vert u\right\Vert _{\mathscr{H}_{E}}^{2}: & =\frac{1}{2}\underset{\left(x,y\right)\in E}{\sum\sum}c_{xy}\left|u\left(x\right)-u\left(y\right)\right|^{2}\label{eq:Enorm}
\end{align}
then $\mathscr{H}_{E}$ is a Hilbert space \cite{JoPe10}.
\begin{lem}
\label{lem:dipole}For all $x,y\in V$, there is a unique real-valued
\uline{dipole} vector $v_{xy}\in\mathscr{H}_{E}$ s.t.
\begin{equation}
\left\langle v_{xy},u\right\rangle _{\mathscr{H}_{E}}=u\left(x\right)-u\left(y\right),\;\forall u\in\mathscr{H}_{E}.\label{eq:dipole}
\end{equation}
\end{lem}
\begin{proof}
By the connectedness assumption (see (\ref{enu:a4}), Section \ref{sub:setting}),
one checks that 
\[
\mathscr{H}_{E}\ni u\longmapsto u\left(x\right)-u\left(y\right)\in\mathbb{C}
\]
is a bounded linear functional on $\mathscr{H}_{E}$; so by Riesz's
theorem there exists a unique $v_{xy}\in\mathscr{H}_{E}$ s.t. (\ref{eq:dipole})
holds. For details, see e.g., \cite{JoPe10,JoPe11a}.
\end{proof}
A difficulty with $\Delta$ is that there is not an independent characterization
of the domain $dom\left(\Delta,\mathscr{H}_{E}\right)$ when $\Delta$
is viewed as an operator in $\mathscr{H}_{E}$ (as opposed to in $l^{2}(V)$);
other than what we do in Definition \ref{def:D}, i.e., we take for
its domain $\mathscr{D}_{E}$ = finite span of dipoles. This creates
an ambiguity with functions on $V$ versus vectors in $\mathscr{H}_{E}$.
Note, vectors in $\mathscr{H}_{E}$ are equivalence classes of functions
on $V$. In fact we will see that it is not feasible to aim to prove
properties about $\Delta$ in $\mathscr{H}_{E}$ without first introducing
dipoles; see Lemma \ref{lem:dipole}. Also the delta-functions $\left\{ \delta_{x}\right\} $
from the $l^{2}(V)$-ONB will typically not be total in $\mathscr{H}_{E}$.
In fact, the $\mathscr{H}_{E}$ ortho-complement of $\left\{ \delta_{x}\right\} $
in $\mathscr{H}_{E}$ consists of the harmonic functions in $\mathscr{H}_{E}$,
see Lemma \ref{lem:Delta} (\ref{enu:D5}).
\begin{defn}
Let $\mathscr{H}$ be a Hilbert space with inner product denoted $\left\langle \cdot,\cdot\right\rangle $,
or $\left\langle \cdot,\cdot\right\rangle _{\mathscr{H}}$ when there
is more than one possibility to consider. Let $J$ be a countable
index set, and let $\left\{ w_{j}\right\} _{j\in J}$ be an indexed
family of non-zero vectors in $\mathscr{H}$. We say that $\left\{ w_{j}\right\} _{j\in J}$
is a \emph{\uline{frame}}\emph{ }for $\mathscr{H}$ iff (Def.)
there are two finite positive constants $b_{1}$ and $b_{2}$ such
that
\begin{equation}
b_{1}\left\Vert u\right\Vert _{\mathscr{H}}^{2}\leq\sum_{j\in J}\left|\left\langle w_{j},u\right\rangle _{\mathscr{H}}\right|^{2}\leq b_{2}\left\Vert u\right\Vert _{\mathscr{H}}^{2}\label{eq:en1}
\end{equation}
holds for all $u\in\mathscr{H}$. We say that it is a \emph{\uline{Parseval}}
frame if $b_{1}=b_{2}=1$. 

For references to the theory and application of \uline{frames},
see e.g., \cite{HJL13,KLZ09,CM13,SD13,KOPT13,EO13}.\end{defn}
\begin{lem}
\label{lem:eframe}If $\left\{ w_{j}\right\} _{j\in J}$ is a Parseval
frame in $\mathscr{H}$, then the (analysis) operator $A=A_{\mathscr{H}}:\mathscr{H}\longrightarrow l^{2}\left(J\right)$,
\begin{equation}
Au=\left(\left\langle w_{j},u\right\rangle _{\mathscr{H}}\right)_{j\in J}\label{eq:en2}
\end{equation}
is well-defined and isometric. Its adjoint $A^{*}:l^{2}\left(J\right)\longrightarrow\mathscr{H}$
is given by 
\begin{equation}
A^{*}\left(\left(\gamma_{j}\right)_{j\in J}\right):=\sum_{j\in J}\gamma_{j}w_{j}\label{eq:en3}
\end{equation}
and the following hold:
\begin{enumerate}
\item The sum on the RHS in (\ref{eq:en3}) is norm-convergent;
\item $A^{*}:l^{2}\left(J\right)\longrightarrow\mathscr{H}$ is co-isometric;
and for all $u\in\mathscr{H}$, we have 
\begin{equation}
u=A^{*}Au=\sum_{j\in J}\left\langle w_{j},u\right\rangle w_{j}\label{eq:en4}
\end{equation}
where the RHS in (\ref{eq:en4}) is norm-convergent. 
\end{enumerate}
\end{lem}
\begin{proof}
The details are standard in the theory of frames; see the cited papers
above. Note that (\ref{eq:en1}) for $b_{1}=b_{2}=1$ simply states
that $A$ in (\ref{eq:en2}) is isometric, and so $A^{*}A=I_{\mathscr{H}}=$
the identity operator in $\mathscr{H}$, and $AA^{*}=$ the projection
onto the range of $A$.\end{proof}
\begin{thm}
\label{thm:eframe}Let $G=\left(V,E,c\right)$ be an infinite network.
Choose an orientation on the edges, denoted by $E^{\left(ori\right)}$.
Then the system of vectors
\begin{equation}
\left\{ w_{xy}:=\sqrt{c_{xy}}v_{xy},\;\ensuremath{\left(xy\right)\in E^{\left(ori\right)}}\right\} \label{eq:en8}
\end{equation}
is a Parseval frame for the energy Hilbert space $\mathscr{H}_{E}$.
For all $u\in\mathscr{H}_{E}$, we have the following representation
\begin{align}
u & =\sum_{\left(xy\right)\in E^{\left(ori\right)}}c_{xy}\left\langle v_{xy},u\right\rangle v_{xy},\;\mbox{and}\label{eq:frep}\\
\left\Vert u\right\Vert _{\mathscr{H}_{E}}^{2} & =\sum_{\left(xy\right)\in E^{\left(ori\right)}}c_{xy}\left|\left\langle v_{xy},u\right\rangle \right|^{2}\label{eq:fnorm}
\end{align}
\end{thm}
\begin{proof}
See \cite{JT14,CH08}.\end{proof}
\begin{rem}
While the vectors $w_{xy}:=\sqrt{c_{xy}}v_{xy}$, $\left(xy\right)\in E^{\left(ori\right)}$,
form a Parseval frame in $\mathscr{H}_{E}$ in the general case, typically
this frame is not an orthogonal basis (ONB) in $\mathscr{H}_{E}$.
\end{rem}

\subsection{The Graph-Laplacian}

Here we prove some technical lemmas for graph Laplacian in the energy
Hilbert space $\mathscr{H}_{E}$ .

Let $G=\left(V,E,c\right)$ be as above; assume $G$ is connected;
i.e., there is a base point $o$ in $V$ such that every $x\in V$
is connected to $o$ via a finite path of edges. 

If $x\in V$, we set 
\begin{equation}
\delta_{x}\left(y\right)=\begin{cases}
1 & \mbox{if }y=x\\
0 & \mbox{if }y\neq x
\end{cases}\label{eq:delx}
\end{equation}

\begin{defn}
\label{def:D}Let $\left(V,E,c,o,\Delta\right)$ be as above. Let
$V':=V\backslash\left\{ o\right\} $, and set
\[
v_{x}:=v_{x,o},\;\forall x\in V'.
\]
Further, let 
\begin{align}
\mathscr{D}_{2} & :=span\left\{ \delta_{x}\:\big|\: x\in V\right\} ,\;\mbox{and}\label{eq:D1}\\
\mathscr{D}_{E} & :=\left\{ \sum\nolimits _{\text{finite}}\xi_{x}v_{x}\:\big|\:\xi_{x}\in\mathbb{C},\; x\in V'\right\} ;\label{eq:D2}
\end{align}
where by ``span'' we mean the set of all \emph{finite} linear combinations.
\end{defn}
Lemma \ref{lem:Delta} below summarizes the key properties of $\Delta$
as an operator, both in $l^{2}(V)$ and in $\mathscr{H}_{E}$. 
\begin{lem}
\label{lem:Delta}The following hold:
\begin{enumerate}
\item $\left\langle \Delta u,v\right\rangle _{l^{2}}=\left\langle u,\Delta v\right\rangle _{l^{2}}$,
$\forall u,v\in\mathscr{D}_{2}$;
\item \label{enu:D2}$\left\langle \Delta u,v\right\rangle _{\mathscr{H}_{E}}=\left\langle u,\Delta v\right\rangle _{\mathscr{H}_{E}},$
$\forall u,v\in\mathscr{D}_{E}$;
\item \label{enu:D3}$\left\langle u,\Delta u\right\rangle _{l^{2}}\geq0$,
$\forall u\in\mathscr{D}_{2}$, and
\item \label{enu:D4}$\left\langle u,\Delta u\right\rangle _{\mathscr{H}_{E}}\geq0$,
$\forall u\in\mathscr{D}_{E}$.
\end{enumerate}

Moreover, we have
\begin{enumerate}[resume]
\item \label{enu:D5}$\left\langle \delta_{x},u\right\rangle _{\mathscr{H}_{E}}=\left(\Delta u\right)\left(x\right)$,
$\forall x\in V$, $\forall u\in\mathscr{H}_{E}$.
\item $\Delta v_{xy}=\delta_{x}-\delta_{y}$, $\forall v_{xy}\in\mathscr{H}_{E}$.
In particular, $\Delta v_{x}=\delta_{x}-\delta_{o}$, $x\in V'=V\backslash\left\{ o\right\} $. 
\item \label{enu:D7}
\[
\delta_{x}\left(\cdot\right)=c\left(x\right)v_{x}\left(\cdot\right)-\sum_{y\sim x}c_{xy}v_{y}\left(\cdot\right),\;\forall x\in V'.
\]

\item \label{enu:D8}
\[
\left\langle \delta_{x},\delta_{y}\right\rangle _{\mathscr{H}_{E}}=\begin{cases}
c\left(x\right)=\sum_{t\sim x}c_{xt} & \mbox{if \ensuremath{y=x}}\\
-c_{xy} & \mbox{if \ensuremath{\left(xy\right)\in E}}\\
0 & \mbox{if }\left(xy\right)\notin E,\; x\neq y
\end{cases}
\]

\end{enumerate}
\end{lem}
\begin{proof}
See \cite{JoPe10,JoPe11a,JT14}. Note that the numbers in (\ref{enu:D8})
are precisely the entries in the $\infty\times\infty$ banded matrix
(\ref{eq:lapm}).
\end{proof}

\section{The Krein Extension}

Fix a conductance function $c$. In this section we turn to some technical
lemmas we will need for the Krein extension of $\Delta\left(=\Delta_{c}\right)$;
see Definition \ref{def:kr} below.

It is known the graph-Laplacian $\Delta$ is automatically essentially
selfadjoint as a densely defined operator in $l^{2}(V)$, but not
as a $\mathscr{H}_{E}$ operator \cite{Jor08,JoPe11b}. Since $\Delta$
defined on $\mathscr{D}_{E}$ is semibounded, it has the Krein extension
$\Delta_{Kre}$ (in $\mathscr{H}_{E}$).
\begin{lem}
\label{lem:FriedDomain}Consider $\Delta$ with $dom\left(\Delta\right):=span\left\{ v_{xy}:x,y\in V\right\} $,
then 
\[
\left\langle \varphi,\Delta\varphi\right\rangle _{\mathscr{H}_{E}}=\sum_{\left(xy\right)\in E}c_{xy}^{2}\left|\left\langle v_{xy},\varphi\right\rangle _{\mathscr{H}_{E}}\right|^{2}.
\]
\end{lem}
\begin{proof}
Suppose $\varphi=\sum\varphi_{xy}v_{xy}\in dom(\Delta)$. Note the
edges are not oriented, and a direct computation shows that
\[
\left\langle \varphi,\Delta\varphi\right\rangle _{\mathscr{H}_{E}}=4\sum_{x,y}\left|\varphi_{xy}\right|^{2}.
\]
 Using the Parseval frames in Theorem \ref{thm:eframe}, we have
the following representation 
\[
\varphi=\sum_{\left(xy\right)\in E}\underset{=:\varphi_{xy}}{\underbrace{\frac{1}{2}c_{xy}\left\langle v_{xy},\varphi\right\rangle _{\mathscr{H}_{E}}}}v_{xy}
\]
Note $\varphi\in span\left\{ v_{xy}:x,y\in V\right\} $, so the above
equation contains a finite sum. 

It follows that 
\[
\left\langle \varphi,\Delta\varphi\right\rangle _{\mathscr{H}_{E}}=4\sum_{\left(xy\right)\in E}\left|\varphi_{xy}\right|^{2}=\sum_{\left(xy\right)\in E}c_{xy}^{2}\left|\left\langle v_{xy},\varphi\right\rangle _{\mathscr{H}_{E}}\right|^{2}
\]
which is the assertion.
\end{proof}
Throughout below, we shall assume that $V$ is infinite and countable. 
\begin{thm}
\label{thm:domF}Let $G=\left(V,E,c\right)$ be an infinite network.
If the deficiency indices of $\Delta\left(=\Delta_{c}\right)$ are
$\left(k,k\right)$, $k>0$, where $dom(\Delta)=span\left\{ v_{xy}\right\} $,
then the Krein extension $\Delta_{Kre}\supset\Delta$ is the restriction
of $\Delta^{*}$ to 
\begin{equation}
dom(\Delta_{Kre}):=\left\{ u\in\mathscr{H}_{E}\:\big|\:\sum\nolimits _{\left(xy\right)\in E}c_{xy}^{2}\left|\left\langle v_{xy},u\right\rangle _{E}\right|^{2}<\infty\right\} .\label{eq:FrieDomain}
\end{equation}
 \end{thm}
\begin{proof}
Follows from Lemma \ref{lem:FriedDomain}, and the characterization
of Krein extensions of semibounded Hermitian operators; see, e.g.,
\cite{DS88b,AG93,RS75}. \end{proof}
\begin{defn}
\label{def:kr}The \emph{Krein extension} $\Delta_{Kre}$ is specified
by a selfadjoint and contractive operator $B$ in $\mathscr{H}_{E}$
satisfying $B\left(\varphi+\Delta\varphi\right)=\varphi$, $\forall\varphi\in\mathscr{D}_{E}$;
and then $\Delta_{Kre}=B^{-1}-I_{\mathscr{H}_{E}}$; noting that $B^{-1}$
is well defined, selfadjoint, but unbounded.\end{defn}
\begin{rem}
We shall return to the Krein extension after Corollary \ref{cor:Fri},
in Theorem \ref{thm:LLs}, and in Corollary \ref{cor:Kr} where we
give an explicit formula for the s.a. contraction $B=B_{Kr}$.
\end{rem}

\subsection{A Factorization via $l^{2}\left(V'\right)$}

We begin with some preliminary lemmas about closable operators in
Hilbert space.
\begin{lem}
\label{lem:Fri}Let $\mathscr{H}_{1}$ and $\mathscr{H}_{2}$ be two
Hilbert spaces with respective inner products $\left\langle \cdot,\cdot\right\rangle _{i}$,
$i=1,2$; let $\mathscr{D}_{i}\subset\mathscr{H}_{i}$, $i=1,2$,
be two dense linear subspaces; and let 
\[
L_{0}:\mathscr{D}_{1}\rightarrow\mathscr{H}_{2},\;\mbox{and }M_{0}:\mathscr{D}_{2}\rightarrow\mathscr{H}_{1}
\]
be linear operators such that 
\begin{equation}
\left\langle L_{0}u,v\right\rangle _{2}=\left\langle u,M_{0}v\right\rangle _{1},\;\forall u\in\mathscr{D}_{1},\forall v\in\mathscr{D}_{2}.\label{eq:F2-1}
\end{equation}
Then both operators $L_{0}$ and $M_{0}$ are closable. The closures
$L=\overline{L_{0}}$, and $M=\overline{M_{0}}$ satisfy 
\begin{equation}
L^{*}\subseteq M\;\mbox{and }M\subseteq L^{*}.\label{eq:F2-2}
\end{equation}
Moreover, equality in (\ref{eq:F2-2}) holds if and only if 
\begin{equation}
\mathscr{N}\left(I_{\mathscr{H}_{2}}+M^{*}L^{*}\right)=\mathscr{N}\left(I_{\mathscr{H}_{1}}+L^{*}M^{*}\right)=0.\label{eq:F2-5}
\end{equation}
\end{lem}
\begin{proof}
Note (\ref{eq:F2-1}) is equivalent to 
\[
L_{0}\subset M_{0}^{*},\;\mbox{and }M_{0}\subset L_{0}^{*}.
\]
Since the adjoints $M_{0}^{*},L_{0}^{*}$ are closed, it follows that
both $L_{0}$ and $M_{0}$ are closable. Below we give a direct argument:

Suppose $w\in\mathscr{H}_{2}$, $u_{n}\in\mathscr{D}_{1}$, satisfying
$\left\Vert L_{0}u_{n}-w\right\Vert _{2}\rightarrow0$, $\left\Vert u_{n}\right\Vert _{1}\rightarrow0$,
then, for all $v\in\mathscr{D}_{2}$, we have 
\[
\left\langle L_{0}u_{n},v\right\rangle _{2}=\left\langle u_{n},M_{0}v\right\rangle _{1}
\]
and as $n\rightarrow\infty$, we get 
\[
\left\langle w,v\right\rangle _{2}=\left\langle 0,M_{0}v\right\rangle _{1}=0.
\]
Since $\mathscr{D}_{2}$ is dense in $\mathscr{H}_{2}$, it follows
that $w=0$. Thus, $L_{0}$ is closable. Similarly, $M_{0}$ is closable
as well. See also Lemma \ref{lem:graph} below.

From (\ref{eq:F2-1}) we conclude that 
\begin{equation}
M\subset L^{*}\left(=L_{0}^{*}\right),\;\mbox{and }L\subset M^{*}\left(=M_{0}^{*}\right).\label{eq:F2-3}
\end{equation}

We give conditions for the inclusions in (\ref{eq:F2-3}) to be equal.
To verify $M=L^{*}$, we must prove that the graph of $M$ is dense
in that of $L^{*}$, i.e., if $w\in dom\left(L^{*}\right)$, satisfying
\begin{equation}
\left\langle \begin{pmatrix}v\\
Mv
\end{pmatrix},\begin{pmatrix}w\\
L^{*}w
\end{pmatrix}\right\rangle _{\mathscr{H}_{2}\oplus\mathscr{H}_{1}}=0,\;\forall v\in\mathscr{D}_{2}\label{eq:F2-4}
\end{equation}
then $w=0$. 

Now (\ref{eq:F2-4}) reads 
\[
\left\langle v,w\right\rangle _{2}+\left\langle Mv,L^{*}w\right\rangle _{1}=0,\;\forall v\in\mathscr{D}_{2}.
\]
Hence $L^{*}w\in dom(M^{*})$, and $M^{*}L^{*}w=-w$; i.e., we arrive
at the implication 
\[
\left(I_{\mathscr{H}_{2}}+M^{*}L^{*}\right)w=0\Longrightarrow w=0.
\]
(\ref{eq:F2-5}) follows from this.
\end{proof}

For the result to the effect that $\Delta$ as an operator in $l^{2}(V)$,
with domain consisting of finitely supported functions, is essentially
selfadjoint we cite \cite{MR2920886,2007PhDT.......216W,Jor08,JoPe10}.
\begin{cor}
\label{cor:Fri}Let $L$ and $M$ be as above, including (\ref{eq:F2-5}),
then
\begin{enumerate}
\item $ML$ is selfadjoint in $\mathscr{H}_{1}$, and 
\item $LM$ is selfadjoint in $\mathscr{H}_{2}$. 
\end{enumerate}
\end{cor}
\begin{proof}
This follows from von Neumann's theorem which states that if $T$
is any closed operator with dense domain, then $T^{*}T$ is selfadjoint.
See \cite{DS88b,RS75}.
\end{proof}
Let $\Delta_{Kre}$ denote the Krein-extension of the operator $\Delta\big|_{\mathscr{D}_{E}}$
defined in Definition \ref{def:D}. We have the following:
\begin{rem}
As an application, we consider the factorization $\Delta_{Kre}=LL^{*}$,
where $\Delta_{Kre}$ is the Krein extension of the graph-Laplacian
in $\mathscr{H}_{E}$. See Theorem \ref{thm:Deltaf} below.

Let $G=\left(V,E\right)$ be a connected network as before, $V$ is
countable infinite. Fix a base point $o\in V$, and set $v_{x}:=v_{xo}$,
for all $x\in V':=V\backslash\left\{ o\right\} $. \end{rem}
\begin{defn}
\label{def:D2lp}Let $\mathscr{D}_{l^{2}}'$ be the dense subspace
in $l^{2}\left(V'\right)$, given by
\begin{equation}
\mathscr{D}_{l^{2}}'=\Big\{(\xi_{x})\in l^{2}(V')\:|\:\mbox{ finite support},\;\mbox{and }\sum\nolimits _{x\in V'}\xi_{x}=0\Big\}.\label{eq:Dl21}
\end{equation}
\end{defn}
\begin{thm}
\label{thm:Deltaf}Let $\left(V,E,c,\Delta\left(=\Delta_{c}\right),\mathscr{H}_{E},\Delta_{Kre}\right)$
be as above. 
\begin{enumerate}
\item \label{enu:domL}Set 
\begin{equation}
L\left(\left(\xi_{x}\right)\right):=\sum\nolimits _{x}\xi_{x}\delta_{x}\in\mathscr{H}_{E};\label{eq:F6}
\end{equation}
then $L:l^{2}\left(V'\right)\longrightarrow\mathscr{H}_{E}$ is a
closable operator with dense domain $\mathscr{D}_{l^{2}}'$; and the
corresponding adjoint operator $L^{*}:\mathscr{H}_{E}\longrightarrow l^{2}\left(V'\right)$
satisfies
\begin{equation}
L^{*}\left(\sum\nolimits _{x\in V'}\xi_{x}v_{x}\right)=\xi\left(=\left(\xi_{x}\right)\right).\label{eq:F7}
\end{equation}

\item \label{enu:LL1}$LL^{*}$ is selfadjoint,
\item \label{enu:.LL2}$LL^{*}=\Delta_{Kre}$; and
\item Using Lemma \ref{lem:Delta} (\ref{enu:D7}) we note that $L$ in
(\ref{eq:F6}) may also be written in the following form:
\begin{equation}
L\left(\xi\right)=\sum\nolimits _{x}\xi_{x}\, c\left(x\right)v_{x}-\sum\nolimits _{y}\left(\sum\nolimits _{x\sim y}\xi_{x}c_{xy}\right)v_{y}.\label{eq:F8}
\end{equation}

\end{enumerate}
\end{thm}
\begin{rem}
\label{rem:Fset}In the proof of Theorem \ref{thm:Deltaf}, we will
use Lemma \ref{lem:Fri} and Corollary \ref{cor:Fri}. (See also \cite{JT14}
for more details about graph Laplacians.)

The following are immediate from our constructions:
\begin{equation}
\xymatrix{l^{2}\left(V'\right)\ar@/^{1pc}/[r]^{L} & \mathscr{H}_{E}\ar@/^{1pc}/[l]_{M} &  & l^{2}\left(V'\right)\ar@/_{1pc}/[r]^{M^{*}} & \mathscr{H}_{E}\ar@/_{1pc}/[l]_{L^{*}}}
\label{eq:c1-1}
\end{equation}
\end{rem}
\begin{itemize}
\item $L:l^{2}\left(V'\right)\longrightarrow\mathscr{H}_{E}$, with dense
domain 
\begin{equation}
L\left(\left(\xi_{x}\right)\right):=\sum\nolimits _{x}\xi_{x}\delta_{x},\;\forall\left(\xi_{x}\right)\in\mathscr{D}_{l^{2}}'\subset l^{2}\left(V'\right),\label{eq:c1-2}
\end{equation}
where $\left(\xi_{x}\right)\in\mathscr{D}_{l^{2}}'$ satisfies $\sum_{x}\xi_{x}=0$;
see Def. \ref{def:D2lp}. 
\item $M:\mathscr{H}_{E}\longrightarrow l^{2}\left(V'\right)$, 
\begin{equation}
M\left(\sum\nolimits _{x}\xi_{x}v_{x}\right):=\left(\xi_{x}\right),\;\forall\sum\nolimits _{x}\xi_{x}v_{x}\in\mathscr{D}_{E}\subset\mathscr{H}_{E};\label{eq:c1-3}
\end{equation}
where $\mathscr{D}_{E}:=\left\{ \sum\nolimits _{\text{finite}}\xi_{x}v_{x}\:\big|\:\xi_{x}\in\mathbb{C},\: x\in V'\right\} $,
see Def. \ref{def:D}.
\item $\mathscr{D}_{l^{2}}'$ is dense in $l^{2}\left(V'\right)$, and $\mathscr{D}_{E}$
is dense in $\mathscr{H}_{E}$. 
\begin{equation}
\xymatrix{\mathscr{D}_{l^{2}}'\ar@{_{(}->}[d]\ar[rdd]\sp(0.35){L} & \mathscr{D}_{E}\ar@{_{(}->}[d]\ar[ldd]\sb(0.7){M}\\
dom(M^{*})\ar@{_{(}->}[d]\ar@/^{1pc}/[ur]\sp(0.7){M^{*}} & dom(L^{*})\ar@{_{(}->}[d]\ar[dl]\sp(0.6){L^{*}}\\
l^{2}\left(V'\right) & \mathscr{H}_{E}
}
\label{eq:c1-4}
\end{equation}
\end{itemize}
\begin{proof}[Proof of Theorem \ref{thm:Deltaf}]

Step 1. 
\begin{equation}
\left\langle L\left(\xi\right),u\right\rangle _{\mathscr{H}_{E}}=\left\langle \xi,Mu\right\rangle _{L^{2}\left(V'\right)}\label{eq:c1-5}
\end{equation}
holds $\forall\xi\in\mathscr{D}_{l^{2}}'\subset l^{2}\left(V'\right)$,
$\forall u\in\mathscr{D}_{E}\subset\mathscr{H}_{E}$. It follows that
\begin{equation}
L\subset M^{*},\;\mbox{and }M\subset L^{*}.\label{eq:c1-6}
\end{equation}
(We define containment of operators as containment of the respective
graphs. We will show that ``equality'' holds in (\ref{eq:c1-6}).)

The proof will resume after the following digression:\end{proof}
\begin{lem}
\label{lem:graph}Let $\mathscr{H}_{1}\supset\mathscr{D}\xrightarrow{\quad T\quad}\mathscr{H}_{2}$
be densely defined, with $dom\left(T\right)=\mathscr{D}$, and let
\[
\mathscr{G}\left(T\right)=\left\{ \left(u,Tu\right):u\in\mathscr{D}\right\} \subset\mathscr{H}_{1}\oplus\mathscr{H}_{2}
\]
be the graph of $T$. Set 
\[
\chi\begin{pmatrix}u\\
v
\end{pmatrix}=\begin{pmatrix}-v\\
u
\end{pmatrix}
\]
for all $\left(u,v\right)\in\mathscr{H}_{1}\oplus\mathscr{H}_{2}$.
Then
\[
\chi\left(\mathscr{G}\left(T\right)\right)^{\perp}=\mathscr{G}\left(T^{*}\right).
\]
See the diagram below. 
\[
\xymatrix{\mathscr{H}_{1}\ar@/^{1pc}/[r]^{T} & \mathscr{H}_{2}\ar@/^{1pc}/[l]^{T^{*}}}
\]
\end{lem}
\begin{proof}
Note the following statements are all equivalent: 
\begin{align*}
\left\langle \begin{pmatrix}-Tu\\
u
\end{pmatrix},\begin{pmatrix}w_{2}\\
w_{1}
\end{pmatrix}\right\rangle  & =0,\;\forall u\in\mathscr{D}\\
 & \Updownarrow\\
\left\langle Tu,w_{2}\right\rangle  & =\left\langle u,w_{1}\right\rangle ,\;\forall u\in\mathscr{D}\\
 & \Updownarrow\\
w_{2}\in dom(T^{*})\; & \mbox{and}\; T^{*}w_{2}=w_{1}\\
 & \Updownarrow\\
\begin{pmatrix}w_{2}\\
w_{1}
\end{pmatrix} & \in\mathscr{G}\left(T^{*}\right).
\end{align*}
\end{proof}
\begin{cor}
$T$ is closable $\Longleftrightarrow$ $dom(T^{*})$ is dense in
$\mathscr{H}_{2}$.
\end{cor}

\begin{cor}
Let $M,M^{*},L,L^{*}$ be as before. Then
\begin{enumerate}
\item $\xi\in dom(M^{*})\Longleftrightarrow\xi\in l^{2}\left(V'\right)$
and $\exists C=C_{\xi}<\infty$ s.t. 
\begin{equation}
\left|\left\langle \xi,Mu\right\rangle _{l^{2}\left(V'\right)}\right|\leq C\left\Vert u\right\Vert _{\mathscr{H}_{E}},\;\forall u\in\mathscr{D}_{E};\label{eq:c1-7}
\end{equation}

\item $u\in dom(L^{*})\Longleftrightarrow u\in\mathscr{H}_{E}$ and $\exists C=C_{u}<\infty$
s.t. 
\begin{equation}
\left|\left\langle L\left(\xi\right),u\right\rangle _{\mathscr{H}_{E}}\right|\leq C\left\Vert \xi\right\Vert _{l^{2}\left(V'\right)},\;\forall\xi\in\mathscr{D}_{l^{2}}'.\label{eq:c1-8}
\end{equation}
 
\end{enumerate}
\end{cor}
\uline{Comments on \mbox{(\ref{eq:c1-7})}, i.e., $dom(M^{*})$}

Let $u=\sum_{x}\eta_{x}v_{x}$, $\sum\eta_{x}=0$, finite support;
then 
\begin{equation}
\left\langle \xi,Mu\right\rangle _{l^{2}\left(V'\right)}=\sum\nolimits _{x}\xi_{x}\eta_{x}\label{eq:c1-9}
\end{equation}
(It suffices to specialize to real valued functions.) Note (\ref{eq:c1-9})
$\Longleftrightarrow$ 
\[
\left|\sum\nolimits _{x}\xi_{x}\eta_{x}\right|\leq C_{\xi}\left\Vert u\right\Vert _{\mathscr{H}_{E}},\;\forall u\in\mathscr{D}_{E};
\]
i.e., $\xi\in dom(M^{*})\subset l^{2}\left(V'\right)$. 

\uline{Comments on \mbox{(\ref{eq:c1-8})}, i.e., $dom(L^{*})$}

Using the same notations as above, we get (\ref{eq:c1-8}) $\Longleftrightarrow$
\[
\left|\sum\nolimits _{x}\xi_{x}\eta_{x}\right|\leq C_{u}\left\Vert \xi\right\Vert _{l^{2}},\;\forall\xi\in\mathscr{D}_{l^{2}\left(V'\right)};
\]
i.e., $u\in dom(L^{*})\subset\mathscr{H}_{E}$. 
\begin{cor}
We have 
\[
u\in dom(L^{*})\Longleftrightarrow\begin{bmatrix}u=\sum_{x}\eta_{x}v_{x}\in\mathscr{H}_{E}\;\mbox{with}\\
\left(\eta_{x}\right)\in l^{2}\left(V'\right),\; i.e.,\sum_{x}\left|\eta_{x}\right|^{2}<\infty.
\end{bmatrix}.
\]

\end{cor}

\begin{proof}[Proof of Theorem \ref{thm:Deltaf}, continued]

Step 2. We have $M\subset L^{*}$ so $dom(L^{*})$ is dense. But we
need $M=L^{*}$ where $M$ is the closure. Turn to ortho-complement:

Fix $\begin{pmatrix}w\\
L^{*}w
\end{pmatrix}\in\mathscr{G}\left(L^{*}\right)=$ graph of $L^{*}$ s.t. 
\begin{align*}
\left\langle \begin{pmatrix}\mathscr{H}_{E}\ni w:=\sum\nolimits _{x}\xi_{x}v_{x}\\
l^{2}\ni\zeta:=L^{*}\left(\sum\nolimits _{x}\xi_{x}v_{x}\right)
\end{pmatrix},\begin{pmatrix}\mathscr{D}_{E}\ni u:=\sum\nolimits _{x}\eta_{x}v_{x}\\
\eta:=M\left(\sum\nolimits _{x}\eta_{x}v_{x}\right)
\end{pmatrix}\right\rangle  & =0,\;\forall\begin{pmatrix}u\\
Mu
\end{pmatrix}\in\mathscr{G}\left(M\right)\\
 & \Updownarrow\\
\left\langle w,u\right\rangle _{\mathscr{H}_{E}}+\sum\nolimits _{x}\zeta_{x}\eta_{x}=0,\;\forall u=\sum\nolimits _{x}\eta_{x}v_{x} & \in\mathscr{D}_{E}\\
 & \Updownarrow\\
\sum\nolimits _{x}\left(w\left(x\right)-w\left(o\right)\right)\eta_{x}+\sum\nolimits _{x}\zeta_{x}\eta_{x}=0,\;\forall u=\sum\nolimits _{x}\eta_{x}v_{x} & \in\mathscr{D}_{E}\\
 & \Updownarrow\\
\zeta_{x}=w\left(o\right)-w\left(x\right) & \in l^{2}\left(V'\right)\\
 & \Downarrow\\
L^{*}w=\xi\;(\mbox{Recall }w=\sum\nolimits _{x}\xi_{x}v_{x}) & \Rightarrow L^{*}=M
\end{align*}
where we identify $M$ with its closure.

The second equality, $M=L^{*}$, is Corollary \ref{cor:MLaj} below.

Details for the rest of the proof are in Lemma \ref{lem:T1}-\ref{lem:T2},
and the corollaries below.\end{proof}
\begin{rem}
For a general symmetric pair of unbounded operators with dense domains,
$L_{0}$ and $M_{0}$ (as above) we have that each will be contained
in the adjoint of the other. So this means that each of the two adjoints
must have dense domain; and so be closable. If the closures are denoted
$L$ and $M$, respectively, then this new pair will still have each
is contained in the adjoint of the other, but in general these inclusions
might be strict. Nonetheless, if one is an equality, then so is the
other; i.e, the conclusion $L^{*}=M$ implies $M^{*}=L$. (See also
Corollary \ref{cor:aj} below.)\end{rem}
\begin{example}
\label{ex:nn}Let $G=\left(V,E,c\right)$, with $V=\mathbb{Z}_{+}\cup\left\{ 0\right\} $;
i.e., a nearest neighbor graph ($x,y\in V$ are connected by an edge
whenever $x$ is a a nearest neighbor of $y$).

Fix $A>1$. Set 
\[
c_{n,n+1}\equiv1,\; c_{n,n+1}^{\left(A\right)}=A^{n};
\]
then $c<c^{\left(A\right)}$ point-wise on $E$. 

If we introduce $\mathscr{H}_{C}$ and $\mathscr{H}_{A}$, where 
\begin{align*}
\mathscr{H}_{C} & =\left\{ u:\sum_{n}\left|u\left(n\right)-u\left(n+1\right)\right|^{2}<\infty\right\} \\
\mathscr{H}_{A} & =\left\{ u:\sum_{n}A^{n}\left|u\left(n\right)-u\left(n+1\right)\right|^{2}<\infty\right\} 
\end{align*}
then 
\[
v_{xy}\left(\cdot\right)=\begin{cases}
v_{xy}\left(x\right) & \mbox{if }n<x\\
v_{xy}\left(y\right) & \mbox{if }n>y\\
v_{xy}\left(x\right)-n & \mbox{if }x<n\leq y
\end{cases}
\]
and 
\[
v_{xy}^{\left(A\right)}\left(n\right)=\begin{cases}
v_{xy}^{\left(A\right)}\left(x\right) & \mbox{if }n<x\\
v_{xy}^{\left(A\right)}\left(y\right) & \mbox{if }n>y\\
v_{xy}^{\left(A\right)}\left(x\right)-\sum_{x<k\leq n}\frac{1}{A^{k}} & \mbox{if }x<n\leq y
\end{cases}
\]
See Fig \ref{fig:dipole} below. 

Moreover, if $j_{A}:\mathscr{H}_{A}\longrightarrow\mathscr{H}_{C}$
is the natural inclusion then 
\[
j_{A}^{*}v_{xy}=v_{xy}^{\left(A\right)}.
\]
\end{example}
\begin{proof}
See Lemma \ref{lem:jdipole}.\end{proof}
\begin{rem}
In $\mathscr{H}_{A}$, we consider the associated Laplacian 
\[
\left(\Delta_{A}f\right)\left(n\right)=A^{n}\left(f\left(n\right)-f\left(n+1\right)\right)+A^{n-1}\left(f\left(n\right)-f\left(n-1\right)\right)
\]
defined initially on the dense domain $\mathscr{D}_{A}=span\{v_{\left(n,n+1\right)}^{\left(A\right)}\}$.
This operator has deficiency indices $\left(1,1\right)$ as one checks
using domination with the geometric series $\sum_{1}^{\infty}A^{-n}$. 
\end{rem}

\begin{rem}[\cite{JT14}]
 In the example, $c_{n,n+1}\equiv1$, the Laplacian $\Delta=\Delta_{c}$
is bounded. 

In the example $c_{n,n+1}^{\left(A\right)}=A^{n}$, $A>1$, the Laplacian
$\Delta_{A}$ is unbounded, and its von Neumann deficiency indices
are $\left(1,1\right)$. Note since $A>1$, so $\left\{ A^{-n}\right\} _{n=1}^{\infty}$
is monotone decreasing, and $\sum_{n\geq k}A^{-n}=\frac{A^{-k+1}}{A-1}$
is convergent.
\end{rem}
\begin{figure}[H]
\begin{tabular}{cc}
\includegraphics[scale=0.5]{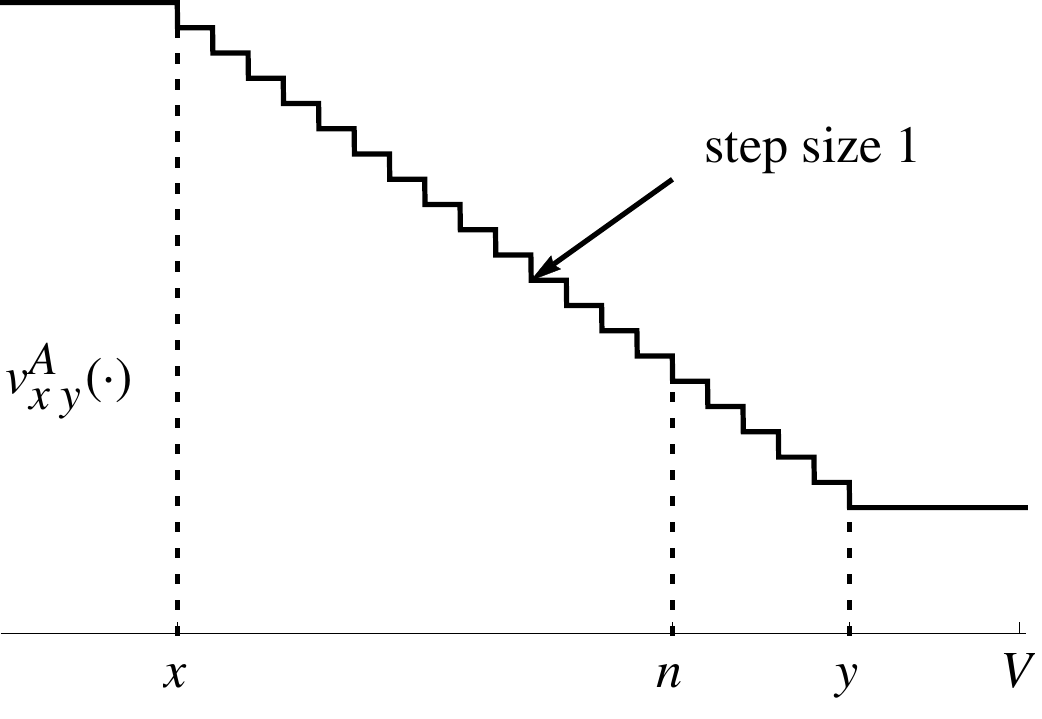} & \includegraphics[scale=0.5]{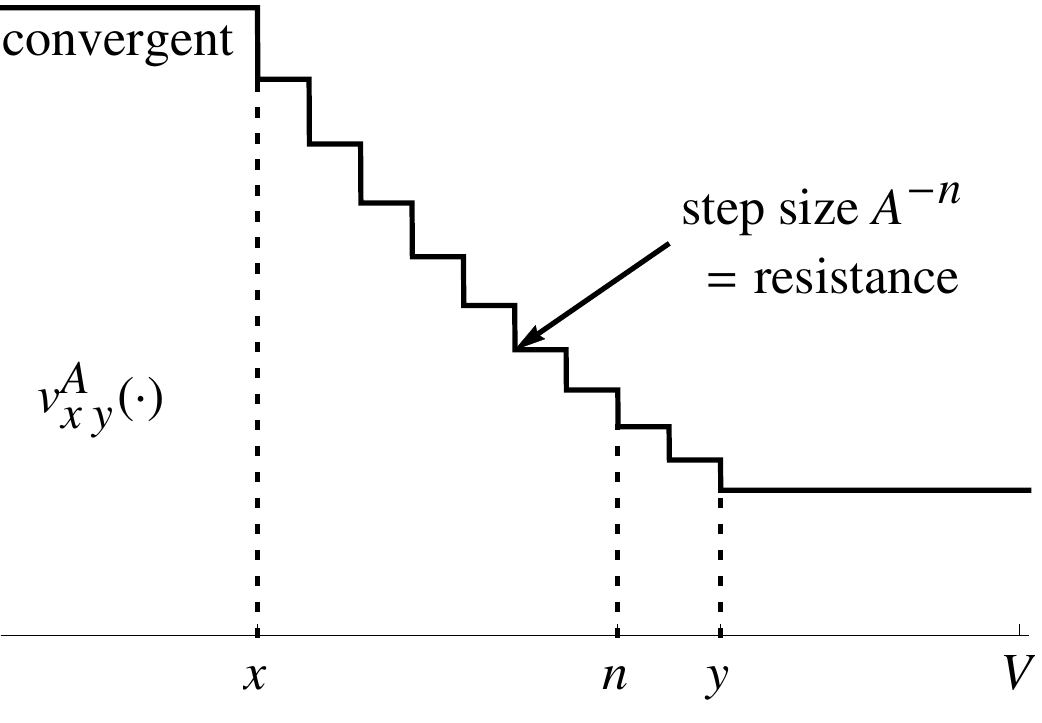}\tabularnewline
unit conductance & conductance $=A^{n}$\tabularnewline
\end{tabular}

\protect\caption{\label{fig:dipole}The dipoles $v_{xy}$ and $v_{xy}^{\left(A\right)}$.}
\end{figure}

\begin{rem}
The metric space $\left(\mathbb{Z}_{+},d_{c}\right)$ is unbounded
when $c\equiv$ unit conductance, while $\left(\mathbb{Z}_{+},d_{A}\right)$
is a bounded metric space when $c_{n,n+1}^{\left(A\right)}=A^{n}$,
$A>1$.
\end{rem}

\subsection{Further Comments}

Fix $\left(V,E,c,o\right)$ be as above. Recall
\begin{itemize}[itemsep=1em]
\item $V$ - vertices, $\#V=\aleph_{0}$, so infinite;
\item $E$ - edges, see section \ref{sub:setting} above;
\item $c$ - conductance, defined on $E$;
\item $o$ - a fixed base-point in $V$;
\item $\mathscr{H}_{E}$ - energy space; 
\item $l^{2}:=l^{2}\left(V'\right)$, $V':=V\backslash\left\{ o\right\} $,
and set $v_{x}:=v_{xo}$, $x\in V'$.
\end{itemize}
Set 
\begin{align}
\mathscr{D}_{E} & =\left\{ \sum\nolimits _{x}\xi_{x}v_{x}\:\big|\:\left(\xi_{x}\right)\:\mbox{finite},\;\sum\nolimits _{x}\xi_{x}=0\right\} \label{eq:fc1}\\
l_{0}^{2}:=\mathscr{D}_{l^{2}}' & =\left\{ \left(\xi_{x}\right)\:\big|\:\left(\xi_{x}\right)\:\mbox{finite support, }\sum\nolimits _{x}\xi_{x}=0\right\} \label{eq:fc2}
\end{align}
$\mathscr{D}_{E}$ is dense in $\mathscr{H}_{E}$, and $l_{0}^{2}$
dense in $l^{2}$. 

Let $L_{0}$ and $M_{0}$ be as follows, see Remark \ref{rem:Fset}:
\[
\xymatrix{l^{2}\ar@/^{1pc}/[r]^{L_{0}} & \mathscr{H}_{E}\ar@/^{1pc}/[l]^{M_{0}}}
\]

Set $dom\left(L_{0}\right):=l_{0}^{2}$, and 
\begin{equation}
L_{0}\xi:=\sum\nolimits _{x}\xi_{x}\delta_{x},\;\forall\xi\in l_{0}^{2};\label{eq:fc3}
\end{equation}

Set $dom\left(M_{0}\right):=\mathscr{D}_{E}$, and 
\begin{equation}
M_{0}u=\xi,\;\forall u:=\sum\nolimits _{x}\xi_{x}v_{x}\in\mathscr{H}_{E}.\label{eq:fc4}
\end{equation}

\begin{lem}
\label{lem:T1}We have 
\[
\left\langle L_{0}\eta,u\right\rangle _{\mathscr{H}_{E}}=\left\langle \eta,M_{0}u\right\rangle _{l^{2}}
\]
for all $\eta\in l_{0}^{2}$, and all $u\in\mathscr{D}_{E}$.\end{lem}
\begin{cor}
\label{cor:aj}$M_{0}\subset L_{0}^{*}$, and $L_{0}\subset M_{0}^{*}$.
Identify $M$ with the closure of the operator defined initially only
on $\mathscr{D}_{E}$; and identify $L$ with the closure of the operator
defined initially only on $l_{0}^{2}$; i.e., 
\[
M:=\overline{M_{0}},\quad L=\overline{L_{0}}.
\]
\end{cor}
\begin{lem}
\label{lem:T2}$L_{0}^{*}=L^{*}$ is as follows: A vector $u\in\mathscr{H}_{E}$
is in $dom(L^{*})\Longleftrightarrow\Delta u\in l^{2}\left(V'\right)$;
see Definition \ref{def:lap}. \end{lem}
\begin{proof}
By definition of the adjoint $L^{*}$, $u\in dom(L^{*})\overset{D}{\Longleftrightarrow}\exists C<\infty$
s.t. 
\begin{equation}
\left|\left\langle L\eta,u\right\rangle _{\mathscr{H}_{E}}\right|\leq C\left\Vert \eta\right\Vert _{l^{2}},\;\forall\eta\in l_{0}^{2};\label{eq:fc8}
\end{equation}
see the diagram below. 
\[
\xymatrix{l^{2}\ar@/_{1pc}/[r]_{L} & \mathscr{H}_{E}\ar@/_{1pc}/[l]_{L^{*}}}
\]
Compute (\ref{eq:fc8}) as follows: Let $\eta\in l_{0}^{2}$, then
\begin{eqnarray*}
\left(\mbox{LHS}\right)_{\left(\ref{eq:fc8}\right)} & = & \left\langle L\eta,u\right\rangle _{\mathscr{H}_{E}}\\
 & \underset{\left(\text{by }\left(\ref{eq:fc3}\right)\right)}{=} & \sum\nolimits _{x}\eta_{x}\left\langle \delta_{x},u\right\rangle _{\mathscr{H}_{E}}\\
 & = & \sum\nolimits _{x}\eta_{x}\left(\Delta u\right)\left(x\right),
\end{eqnarray*}
where the last equality follows from Lemma \ref{lem:Delta} (\ref{enu:D5}).
So if (\ref{eq:fc8}) holds for some $C<\infty$, then $\left(\Delta u\right)\left(\cdot\right)\in l^{2}$,
and $\left(L^{*}u\right)\left(x\right)=\left(\Delta u\right)\left(x\right)$,
so 
\begin{equation}
\left\langle L\eta,u\right\rangle _{\mathscr{H}_{E}}=\sum\nolimits _{x}\eta_{x}\left(\Delta u\right)\left(\cdot\right)=\left\langle \eta,L^{*}u\right\rangle _{l^{2}},\label{eq:fc9}
\end{equation}
i.e., $L^{*}u=\left(\Delta u\right)\left(\cdot\right)$.\end{proof}
\begin{cor}
\label{cor:MLaj}
\begin{equation}
M=L^{*}.\label{eq:fc10}
\end{equation}
\end{cor}
\begin{proof}
We have (\ref{eq:fc10}) $M\subset L^{*}$ by Cor. \ref{cor:aj}.
So $\mathscr{G}\left(M\right)\subset\mathscr{G}\left(L^{*}\right)$,
i.e., containment of operator graphs. Now compute $\mathscr{G}\left(L^{*}\right)\ominus\mathscr{G}\left(M\right)$:
(see (\ref{eq:fc3})-(\ref{eq:fc4}))

We have the graph inner product in $\mathscr{H}_{E}\oplus l^{2}$
\begin{equation}
\left\langle \underset{\in\mathscr{G}\left(M\right)}{\underbrace{\begin{pmatrix}\sum\xi_{x}v_{x}\in\mathscr{H}_{E}\\
\downarrow\\
\left(\xi_{x}\right)\in l^{2}
\end{pmatrix}}},\underset{\in\mathscr{G}\left(L^{*}\right)}{\underbrace{\begin{pmatrix}u\in\mathscr{H}_{E}\\
\downarrow\\
\left(\Delta u\right)\left(x\right)\in l^{2}
\end{pmatrix}}}\right\rangle \label{eq:fc12}
\end{equation}

Assume $\begin{pmatrix}u\\
L^{*}u
\end{pmatrix}\in\mathscr{G}\left(M\right)^{\perp}$. Claim $u=0$. 

\uline{Proof of the claim:} It is equivalent to prove
\begin{equation}
\left\langle \sum\nolimits _{x}\xi_{x}v_{x},u\right\rangle _{\mathscr{H}_{E}}+\left\langle \left(\xi_{x}\right),L^{*}u\right\rangle _{l^{2}}=0\label{eq:fc13}
\end{equation}
for all $\left(\xi_{x}\right)$ finite supported, s.t. $\sum_{x}\xi_{x}=0$.
See (\ref{eq:fc4}). Now rewrite (\ref{eq:fc13}) as follows
\begin{equation}
0=\sum\nolimits _{x}\xi_{x}\left\langle v_{x},u\right\rangle _{\mathscr{H}_{E}}+\sum\nolimits _{x}\xi_{x}\left(\Delta u\right)\left(x\right)\label{eq:fc14}
\end{equation}
for all $\left(\xi_{x}\right)$ finite support, $\sum\xi_{x}=0$;
and use 
\[
\left\langle v_{x},u\right\rangle _{\mathscr{H}_{E}}=u\left(x\right)-u\left(o\right)
\]
\[
\sum\nolimits _{x}\xi_{x}\left(u\left(x\right)-u\left(o\right)\right)=\sum\nolimits _{x}\xi_{x}u\left(x\right),\;\mbox{since }\sum\nolimits _{x}\xi_{x}=0;
\]
then (\ref{eq:fc14}) $\Longleftrightarrow$ 
\[
\sum\nolimits _{x}\xi_{x}\left\{ u\left(x\right)+\left(\Delta u\right)\left(x\right)\right\} =0
\]
for all finite support $\left(\xi_{x}\right)$ s.t. $\sum_{x}\xi_{x}=0$.
But then we get 
\[
\mathscr{H}_{E}\ni u=-\Delta u\in l^{2}
\]
and therefore $u\in\mathscr{H}_{E}\cap l^{2}$, and $\left\langle u,\Delta u\right\rangle _{\mathscr{H}_{E}}\geq0$
by Lemma \ref{lem:Delta}. So ``$-\left\Vert u\right\Vert _{\mathscr{H}_{E}}^{2}\geq0$''
$\Longrightarrow$ $u=0$.
\end{proof}

\section{\label{sec:sp}The Spectrum of the Graph Laplacians}

In this section we show the following result (Theorem \ref{thm:LLs}):
Starting with a fixed graph $(V,E)$ and a fixed conductance function
$c$, we arrive at two versions of a selfadjoint graph Laplacian,
one defined naturally in the $l^{2}$ space of $V$, and the other
in the energy Hilbert space defined from $c$. We prove that, as sets,
the two spectra are the same, aside from the point $0$. The point
zero may be in the spectrum of the second, but not the first. In addition
to this theorem, we isolate other spectral similarities. In section
\ref{sec:vc} below we then study how the spectrum changes subject
to variations in choice of conductance function $c$.

Let $\left(V,E,c,o,\Delta\left(=\Delta_{c}\right),\mathscr{H}_{E}\right)$
be as above. As in Definition  \ref{def:D2lp}, we introduce the operator
$L:l^{2}\longrightarrow\mathscr{H}_{E}$ where $l^{2}:=l^{2}\left(V'\right)$,
$V':=V\backslash\left\{ o\right\} $. Recall $L$ is defined initially
as $L_{0}$; and 
\begin{equation}
dom\left(L_{0}\right)=\left\{ \left(\xi_{x}\right)\:\big|\:\mbox{finitely supported},\;\sum\nolimits _{x}\xi_{x}=0\right\} ;\label{eq:ls1}
\end{equation}
and we let 
\begin{equation}
L:=\overline{L_{0}}\label{eq:ls2}
\end{equation}
i.e., the operator arises as the closure of $L_{0}$
\begin{align}
\mathscr{G}\left(L\right) & =\overline{\mathscr{G}\left(L_{0}\right)}\;\left(\mbox{norm closure}\right)\;\mbox{where}\label{eq:ls3}\\
\mathscr{G}\left(L_{0}\right) & =\left\{ \left(\xi,L_{0}\xi\right)\:\big|\:\xi\in dom\left(L_{0}\right)\right\} \subset l^{2}\oplus\mathscr{H}_{E}.\label{eq:ls3-1}
\end{align}
(For the discussion of closed operators, see \cite{DS88b,RS75}.)

We assume throughout the conductance $c$ is fixed, and that $\left(V,E,c\right)$
is \emph{\uline{connected}}. Also, we fix a base-point $o$, and
set $v_{x}:=v_{xo}$ (the $\left(xo\right)$ dipole for $x\in V'$),
and we recall that 
\begin{equation}
L_{0}\left(\xi\right):=\sum\nolimits _{x}\xi_{x}\delta_{x}\in\mathscr{H}_{E},\;\mbox{and}\label{eq:ls4}
\end{equation}
\begin{equation}
L_{0}^{*}\left(\sum\nolimits _{x}\xi_{x}v_{x}\right):=\left(\xi_{x}\right)\in l^{2}\left(V'\right).\label{eq:ls5}
\end{equation}

We shall compute the spectrum of the following two versions of the
graph Laplacian: 
\begin{equation}
\left(\Delta u\right)\left(x\right)=\sum_{y\sim x}c_{xy}\left(u\left(x\right)-u\left(y\right)\right).\label{eq:ls6}
\end{equation}

\begin{enumerate}[label=\bf V\arabic{enumi}.]
\item  $\Delta$ is viewed as a selfadjoint operator in the Hilbert space
$l^{2}\left(V'\right)$
\item Refers to $\Delta$ as the selfadjoint Krein extension $\Delta_{Kr}$
in $\mathscr{H}_{E}$. 
\end{enumerate}
For the operator in version 1, we shall write $\Delta_{l^{2}}$ for
identification; and 
\begin{equation}
\begin{split}dom\left(\Delta_{l^{2}}\right)=\left\{ \xi\in l^{2}\:\big|\:\Delta\xi\in l^{2}\right\}  & \mbox{ where}\\
\left(\Delta\xi\right)\left(x\right):=\sum_{y\sim x}c_{xy}\left(\xi_{x}-\xi_{y}\right)
\end{split}
\label{eq:ls7}
\end{equation}

It is known \cite{Jor08,JoPe11b} that the operator $\Delta_{l^{2}}$
is selfadjoint in $l^{2}$; generally unbounded; i.e., its deficiency
indices are $\left(0,0\right)$ referring to $l^{2}$. 
\begin{thm}
\label{thm:LLs}Let $\Delta_{l^{2}}$ and $\Delta_{Kr}$ be the selfadjoint
graph-Laplacians in $l^{2}\left(V'\right)$ and $\mathscr{H}_{E}$
respectively. Then, 
\begin{equation}
\mbox{spectrum}\left(\Delta_{l^{2}}\right)=\mbox{spectrum}\left(\Delta_{Kr}\right)\backslash\left\{ 0\right\} ;\label{eq:ls8}
\end{equation}
i.e., the spectrum of the two operators agree except for the point
$0$. In particular, one of the two operators is bounded iff the other
is; and then the bounds agree.
\end{thm}
We shall need the following lemma.
\begin{lem}
\label{lem:.Deltal2}$\Delta_{l^{2}}=L^{*}L$.\end{lem}
\begin{proof}
Using (\ref{eq:ls4}), we note that it suffices to prove that for
all $\xi\in dom(L_{0})$ (eq. (\ref{eq:ls1})) we have 
\begin{equation}
\left\langle \xi,\Delta\xi\right\rangle _{l^{2}}=\left\Vert L\xi\right\Vert _{\mathscr{H}_{E}}^{2}=\left\langle \xi,L^{*}L\xi\right\rangle _{l^{2}}.\label{eq:ls10}
\end{equation}
But a direct computation yields (\ref{eq:ls10}); indeed
\begin{align*}
\left(\mbox{LHS}\right)_{\left(\ref{eq:ls10}\right)} & =\sum_{x}\xi_{x}\sum_{y\sim x}c_{xy}\left(\xi_{x}-\xi_{y}\right)\\
 & =\sum_{x}\xi_{x}^{2}\, c\left(x\right)-\underset{\left(xy\right)\in E_{dir}}{\sum\sum}c_{xy}\xi_{x}\xi_{y},
\end{align*}
where $c\left(x\right)=\sum_{y\sim x}c_{xy}$.

Now using 
\begin{equation}
\left\langle \delta_{x},\delta_{y}\right\rangle _{\mathscr{H}_{E}}=\begin{cases}
c\left(x\right) & \mbox{if }x=y\\
-c_{xy} & \mbox{if }\left(xy\right)\in E\\
0 & \mbox{otherwise}
\end{cases};\label{eq:ls11}
\end{equation}
we get 
\begin{eqnarray*}
\left(\mbox{LHS}\right)_{\left(\ref{eq:ls10}\right)} & \underset{\left(\text{by }\left(\ref{eq:ls11}\right)\right)}{=} & \sum\nolimits _{x}\sum\nolimits _{y}\xi_{x}\xi_{y}\left\langle \delta_{x},\delta_{y}\right\rangle _{\mathscr{H}_{E}}\\
 & = & \left\Vert \sum\nolimits _{x}\xi_{x}\delta_{x}\right\Vert _{\mathscr{H}_{E}}^{2}\\
 & \underset{\left(\text{by }\left(\ref{eq:ls4}\right)\right)}{=} & \left\Vert L\xi\right\Vert _{\mathscr{H}_{E}}^{2}=\left(\mbox{RHS}\right)_{\left(\ref{eq:ls10}\right)}
\end{eqnarray*}
which is the desired conclusion.
\end{proof}

\begin{proof}[Proof of Theorem \ref{thm:LLs} ]
From Theorem \ref{thm:Deltaf}, we have $\Delta_{Kr}=LL^{*}$. Using
the polar decomposition \cite{DS88b,RS75} (polar factorization) for
the closed operator $L:l^{2}\rightarrow\mathscr{H}_{E}$, we get a
unique isometry $U:l^{2}\rightarrow\mathscr{H}_{E}$ such that 
\begin{equation}
L=U\left(L^{*}L\right)^{\frac{1}{2}}=\left(LL^{*}\right)^{\frac{1}{2}}U;\label{eq:ls12}
\end{equation}
and therefore by Theorem \ref{thm:Deltaf} and Lemma \ref{lem:.Deltal2},
\begin{equation}
L=U\left(\Delta_{l^{2}}\right)^{\frac{1}{2}}=\left(\Delta_{Kr}\right)^{\frac{1}{2}}U.\label{eq:ls13}
\end{equation}
Note that the two operators in (\ref{eq:ls13}) under the square root
are selfadjoint (and semibounded). 

$\Delta_{l^{2}}\geq0$ is selfadjoint in $l^{2}$; and $\Delta_{Kr}\geq0$
is selfadjoint in $\mathscr{H}_{E}$. Hence the conclusion in Theorem
\ref{thm:LLs} is a consequence of the polar decomposition in (\ref{eq:ls12}),
i.e., that $U$ is isometric in the ortho-complement of the kernel
of $\Delta_{l^{2}}$. The final space $UU^{*}$ is $\mathscr{H}_{E}\ominus\mbox{Ker}\left(\Delta_{F}\right)$.\end{proof}
\begin{cor}
\label{cor:Kr}The contractive selfadjoint operator $B_{Kr}:\mathscr{H}_{E}\rightarrow\mathscr{H}_{E}$
from Definition \ref{def:kr} giving the Krein extension $\Delta_{Kr}$
is as follows:
\begin{equation}
B_{Kr}=U\left(\Delta_{l^{2}}+I_{l^{2}}\right)^{-1}U^{*}\label{eq:Bkr}
\end{equation}
where $U$ is the isometry in (\ref{eq:ls12}).\end{cor}
\begin{proof}
Follows from the unitary equivalence assertion in Theorem \ref{thm:LLs}.
\end{proof}
We reference here Krein\textquoteright s theory of semibounded operators
and their semibounded selfadjoint extensions. In brief outline, Krein
developed a theory (see e.g., \cite{MR2092324,MR1618628,MR0087060})
for all the selfadjoint extensions of semibounded Hermitian operators
such that the selfadjoint extensions preserve the same lower bound.
(There are other selfadjoint extensions, not with the same lower bound,
but we shall not be concerned with them here.) Krein showed that the
family of selfadjoint extensions (with the same lower bound) is in
bijective correspondence with a certain family of contractive selfadjoint
operators, which in turn has a natural order such that the two cases,
the Krein extension, and the Friedrichs extension, are the extreme
ends in the associated \textquotedblleft order-interval.\textquotedblright{}
From our construction, we note that our particular s.a. extension
in the Hilbert space $\mathscr{H}_{E}$ is in fact the Krein-extension.

From general theory we have that $U$ in (\ref{eq:ls12}) is a partial
isometry with initial space $U^{*}U=\left(\ker\left(L^{*}L\right)\right)^{\perp}=\left(\ker\left(\Delta_{l^{2}}\right)\right)^{\perp}$,
but $\ker\left(\Delta_{l^{2}}\right)=0$ on account of assumption
(\ref{enu:a4}) in Section \ref{sub:setting}, i.e., connectedness.
This is an application of a maximum principle for $\Delta_{l^{2}}$.
Recall $\Delta=c\left(I-P\right)$ where 
\begin{align*}
\left(P\xi\right)\left(x\right) & =\sum_{y\sim x}p_{xy}\xi\left(y\right),\quad p_{xy}=c_{xy}/c\left(x\right);\;\mbox{and}\\
\Delta\xi & =0\Longleftrightarrow P\xi=\xi;
\end{align*}
and $P$ is a Markov-operator.

The fact that $LL^{*}$ (as a selfadjoint operator in $\mathscr{H}_{E}$)
is the Krein extension of $\Delta_{E}$ from Lemma \ref{lem:Delta}
(\ref{enu:D4}) follows from the following:

If $f\in\mathscr{H}_{E}$, and 
\begin{equation}
\left\langle f,L\delta_{x}\right\rangle _{\mathscr{H}_{E}}=0,\quad\mbox{for}\;\forall x\in V,\label{eq:kr1}
\end{equation}
then $f\in dom\left(L^{*}\right)$, and $L^{*}f=0$. Indeed, from
(\ref{eq:kr1}), we have 
\[
0=\left\langle L\delta_{x},f\right\rangle _{\mathscr{H}_{E}}=\left\langle \delta_{x},L^{*}f\right\rangle _{l^{2}}=\left(L^{*}f\right)\left(x\right).
\]

\section{\label{sec:vc}Variation of Conductance}

In this section we study how the spectrum of the graph Laplacian $\Delta_{c}$
changes subject to variations in choice of conductance function $c$.
We prove (Theorems \ref{thm:cFri} and \ref{thm:Deltasp}) that the
natural order of conductance functions, i.e., point-wise as functions
on $E$, induces a certain similarity of the corresponding (Krein
extensions of the) two graph Laplacians. Since the spectra are typically
continuous, fine-structure of spectrum must be defined in terms of
equivalence classes of positive Borel measures on the real line. Hence
our detailed comparison of spectra must be phrased involving these;
see Definition \ref{def:spmeas}.

In this section we turn to the details on the comparison the Krein
extensions of $\Delta_{c}$ as the conductance function varies.

Let $G=\left(V,E\right)$ be a network with vertices $V$ and edges
$E$. We assume $V$ is countable infinite, and $G$ is \emph{connected},
i.e., for all $x,y$ in $V$, $\exists$ a finite path $\left\{ e_{i}=\left(x_{i}x_{i+1}\right)\right\} $
in $E$ s.t. $x_{0}=x$, $x_{n}=y$. See Fig \ref{fig:fp}.

\begin{figure}[H]
\includegraphics[scale=0.4]{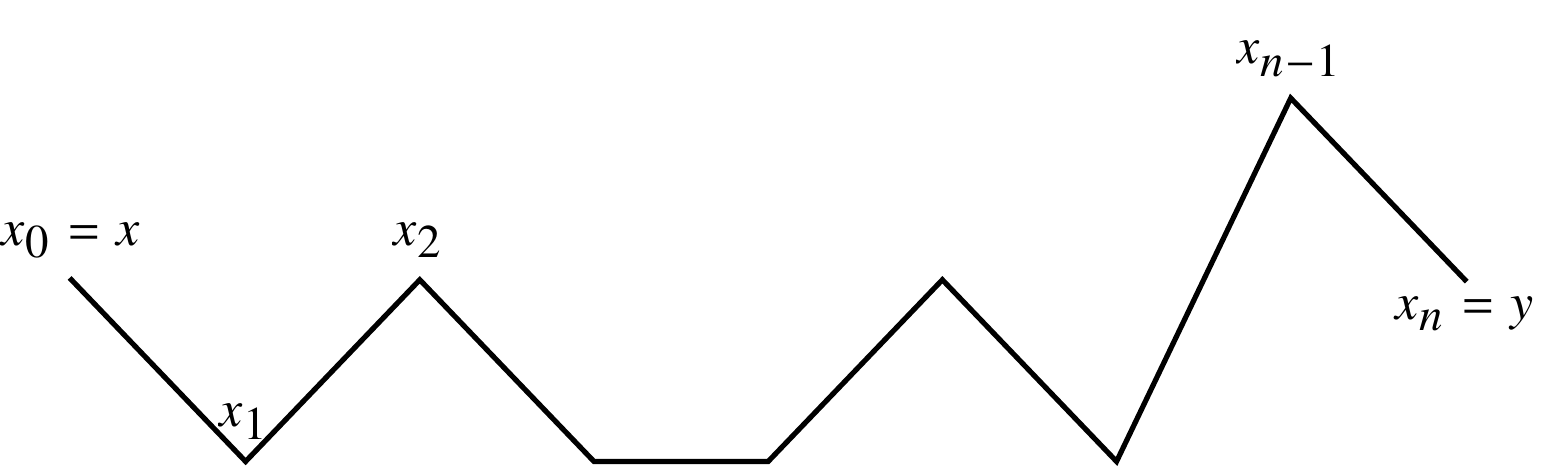}

\protect\caption{\label{fig:fp}A finite path connecting $x,y$, $\left(x_{i}x_{i+1}\right)\in E$,
$c_{x_{i}x_{i+1}}>0$, $i=1,\ldots,n-1$. }

\end{figure}

Given a conductance function $c$, let $\mathscr{H}_{C}$ be the energy
Hilbert space introduced in section \ref{sub:EnergyH}; where 
\begin{align}
\left\langle u,v\right\rangle _{C} & :=\frac{1}{2}\underset{\left(xy\right)\in E}{\sum\sum}c_{xy}(\overline{u\left(x\right)}-\overline{u\left(y\right)})\left(v\left(x\right)-v\left(y\right)\right)\label{eq:Enorm-1}\\
\left\Vert u\right\Vert _{C}^{2}: & =\frac{1}{2}\underset{\left(xy\right)\in E}{\sum\sum}c_{xy}\left|u\left(x\right)-u\left(y\right)\right|^{2}.\label{eq:Einner-1}
\end{align}

Let $c^{A}$ be another conductance function on $E$, and $\mathscr{H}_{A}$
be the corresponding energy Hilbert space. If $c\leq c^{A}$, i.e.,
$c_{xy}\leq c_{xy}^{A}$ for all $\left(xy\right)\in E$, then by
(\ref{eq:Einner-1}), we have 
\begin{equation}
\left\Vert u\right\Vert _{C}^{2}\leq\left\Vert u\right\Vert _{A}^{2},\;\forall u\in\mathscr{H}_{A};\label{eq:c4}
\end{equation}
and so $\mathscr{H}_{A}\subset\mathscr{H}_{C}$. 

\begin{figure}[H]
\begin{tabular}{cc}
\includegraphics[scale=0.4]{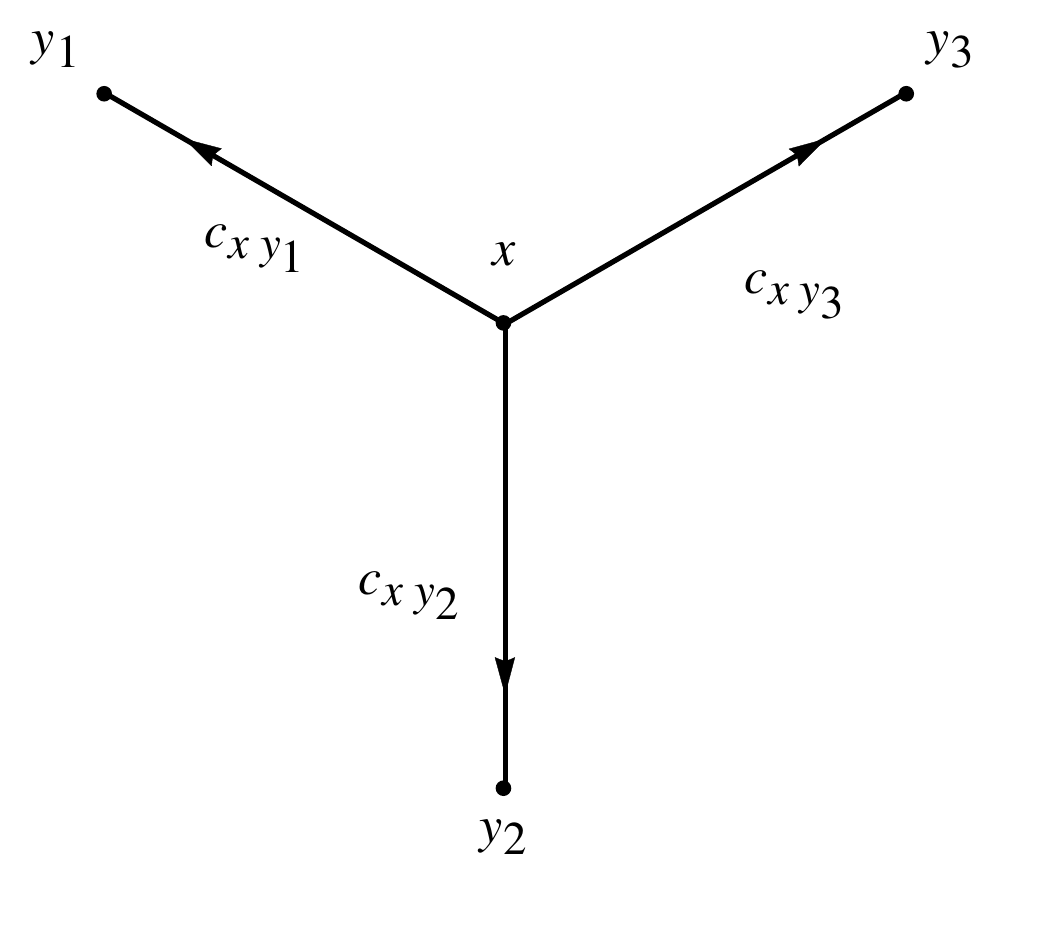} & \includegraphics[scale=0.4]{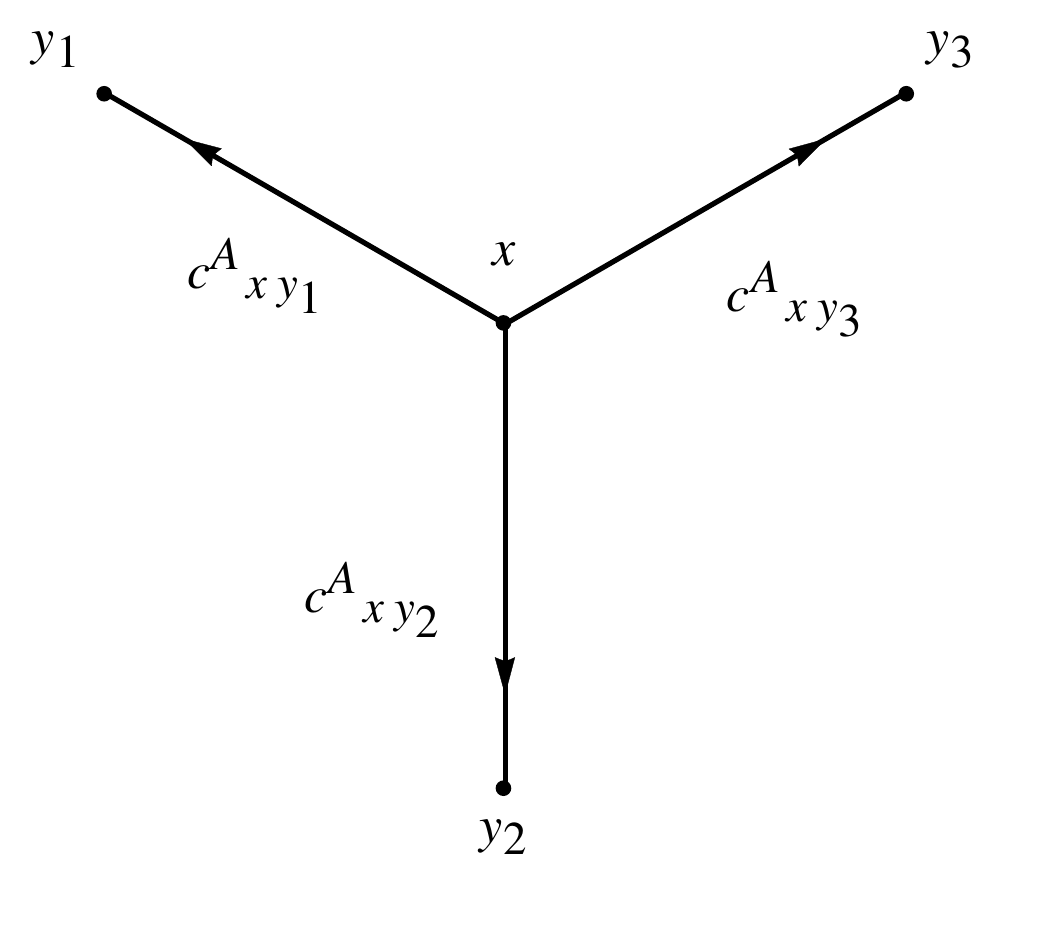}\tabularnewline
\end{tabular}

\protect\caption{$c^{A}\geq c$ point-wise on $E$.}

\end{figure}

\begin{defn}
\label{def:j}Let 
\[
j_{A}:\mathscr{H}_{A}\longrightarrow\mathscr{H}_{C},\quad j_{A}u=u,\;\forall u\in\mathscr{H}_{A}
\]
be the natural inclusion mapping. It follows from (\ref{eq:c4}) that
$j_{A}$ is a continuous (contractive) inclusion.\end{defn}
\begin{lem}
\label{lem:jdipole}Let $c,c^{A},\mathscr{H}_{C},\mathscr{H}_{A}$
and $j_{A}$ be as before. 
\begin{enumerate}
\item Then, 
\begin{equation}
v_{xy}^{\left(A\right)}=j_{A}^{*}v_{xy}\label{eq:jdp}
\end{equation}
where $v_{xy}$ and $v_{xy}^{\left(A\right)}$ are the dipoles in
the respective energy spaces satisfying 
\begin{equation}
\begin{split}\begin{cases}
\left\langle v_{xy},u\right\rangle _{C}=u\left(x\right)-u\left(y\right) & \forall u\in\mathscr{H}_{C}\\
\langle v_{xy}^{\left(A\right)},u\rangle_{A}=u\left(x\right)-u\left(y\right) & \forall u\in\mathscr{H}_{A}\subset\mathscr{H}_{C}
\end{cases}\end{split}
\label{eq:c7}
\end{equation}
(See Lemma \ref{lem:dipole}, and Fig \ref{fig:dipole}  for an illustration.)
\item 
\begin{equation}
j_{A}\left(\delta_{x}^{\left(A\right)}\right)=\delta_{x},\;\forall x\in V.\label{eq:jdel}
\end{equation}

\item 
\begin{equation}
j_{A}^{*}\left(\delta_{x}\right)=\sum_{y\sim x}c_{xy}v_{xy}^{\left(A\right)},\;\forall x\in V.\label{eq:jdelta1}
\end{equation}

\item The following diagram commutes:
\[
\xymatrix{\mathscr{H}_{A}\ar[d]_{\Delta_{A}} & \mathscr{H}_{C}\ar[d]^{\Delta\left(=\Delta_{C}\right)}\ar[l]_{j_{A}^{*}}\\
\mathscr{H}_{A}\ar[r]_{j_{A}} & \mathscr{H}_{C}
}
\]
i.e., 
\begin{equation}
j_{A}\Delta_{A}j_{A}^{*}=\Delta\label{eq:jlap}
\end{equation}
valid on $span\left\{ v_{xy}:x,y\in V\right\} $.
\end{enumerate}
\end{lem}
\begin{proof}
Below we consider pairs of vertices $x,y$, and $s,t$ as follows:
\begin{enumerate}
\item Let $u\in\mathscr{H}_{A}$, then 
\[
u\left(x\right)-u\left(y\right)=\left\langle v_{xy},j_{A}u\right\rangle _{C}=\left\langle j_{A}^{*}v_{xy},u\right\rangle _{A};
\]
which implies $j_{A}^{*}v_{xy}=v_{xy}^{\left(A\right)}$. 
\item For all $v_{st}\in\mathscr{H}_{C}$, 
\begin{align*}
\left\langle j_{A}\delta_{x}^{\left(A\right)},v_{st}\right\rangle _{C} & =\left\langle \delta_{x}^{\left(A\right)},j_{A}^{*}v_{st}\right\rangle _{A}=\left\langle \delta_{x}^{\left(A\right)},v_{st}^{\left(A\right)}\right\rangle _{A}\\
 & =\delta_{s,x}^{\left(A\right)}-\delta_{t,x}^{\left(A\right)}=\delta_{s,x}-\delta_{t,x}=\left\langle \delta_{x},v_{st}\right\rangle _{C}
\end{align*}
and (\ref{eq:jdel}) follows.
\item For all $u\in\mathscr{H}_{A}$, we have 
\begin{align*}
\left\langle j_{A}^{*}\delta_{x},u\right\rangle _{A} & =\left\langle \delta_{x},u\right\rangle _{c}=\left(\Delta_{c}u\right)(x)\\
 & =\sum_{y\sim x}c_{xy}\left(u\left(x\right)-u\left(y\right)\right)\\
 & =\left\langle \sum\nolimits _{y\sim x}c_{xy}v_{xy}^{\left(A\right)},u\right\rangle _{A};
\end{align*}
eq. (\ref{eq:jdelta1}) follows from this.
\item Let $v_{xy}\in\mathscr{H}_{C}$, then 
\begin{eqnarray*}
j_{A}\Delta_{A}j_{A}^{*}v_{xy} & \underset{\left(\text{by }\left(\ref{eq:jdp}\right)\right)}{=} & j\Delta_{A}v_{xy}^{\left(A\right)}\\
 & \underset{\text{lem \ref{lem:Delta}}}{=} & j_{A}\left(\delta_{x}^{\left(A\right)}-\delta_{y}^{\left(A\right)}\right)\\
 & \underset{\left(\text{by }\left(\ref{eq:jdel}\right)\right)}{=} & \delta_{x}-\delta_{y}\\
 & \underset{\text{lem \ref{lem:Delta}}}{=} & \Delta_{C}v_{xy}
\end{eqnarray*}
which gives (\ref{eq:jlap}).
\end{enumerate}
\end{proof}

Let $\left(V,E\right)$ be as above, and $\mathscr{H}_{C}$ and $\mathscr{H}_{A}$
be the energy Hilbert spaces. Fix a base point $o\in V$, and set
$v_{x}:=v_{xo}$, $v_{x}^{A}:=v_{xo}^{A}$, for all $x\in V':=V\backslash\left\{ o\right\} $.

Let 
\begin{align*}
G\left(x,y\right) & =\left\langle v_{x},v_{y}\right\rangle _{\mathscr{H}}\\
G^{\left(A\right)}\left(x,y\right) & =\left\langle v_{x},v_{y}\right\rangle _{\mathscr{H}_{A}}
\end{align*}
be the respective Gramians. 
\begin{lem}
If $c\leq c^{A}$ on $E$, let $j_{A}:\mathscr{H}_{A}\rightarrow\mathscr{H}$
be the natural inclusion as before.
\begin{enumerate}[label=(\roman{enumi}),ref=\roman{enumi}]
\item \label{enu:jac1}Then 
\[
\left\langle v_{x},\left(j_{A}j_{A}^{*}\right)v_{y}\right\rangle _{\mathscr{H}_{C}}=G^{\left(A\right)}\left(x,y\right)
\]
for all $x,y$ in $V'$; and 
\item \label{enu:jac2}We have $\Delta_{y}G_{xy}=\delta_{xy}$.
\end{enumerate}
\end{lem}
\begin{proof}
For (\ref{enu:jac1}), we have 
\begin{align*}
\left\langle v_{x},j_{A}j_{A}^{*}v_{y}\right\rangle _{\mathscr{H}} & =\left\langle j_{A}^{*}v_{x},j_{A}^{*}v_{y}\right\rangle _{\mathscr{H}_{A}}\\
 & =\left\langle v_{x}^{A},v_{y}^{A}\right\rangle _{\mathscr{H}_{A}}=G^{\left(A\right)}\left(x,y\right).
\end{align*}
See the diagram below. 
\[
\xymatrix{\mathscr{H}_{A}\ar@/^{1pc}/[r]^{j_{A}} & \mathscr{H}\ar@/^{1pc}/[l]^{j_{A}^{*}}\ar@(ur,rd)^{j_{A}j_{A}^{*}}}
\]

\uline{Proof of \mbox{(\ref{enu:jac2})}.} We have the following
facts for the Gramian $G\left(x,y\right)=\left\langle v_{x},v_{y}\right\rangle _{\mathscr{H}}$:
\begin{enumerate}
\item $G\left(x,y\right)=v_{y}\left(x\right)-v_{y}\left(o\right)$, where
$x,y\in V'=V\backslash\left\{ o\right\} $, and $o$ is a fixed chosen
base point;
\item $\Delta_{y}G_{xy}=\delta_{xy}$.
\end{enumerate}
\end{proof}
\begin{cor}
The following holds:
\[
v_{x}^{A}\left(y\right)-v_{x}^{A}\left(o\right)=\left(j_{A}j_{A}^{*}\right)\left(v_{x}\right)\left(y\right)-\left(j_{A}j_{A}^{*}\right)\left(v_{x}\right)\left(o\right).
\]

\end{cor}

\begin{cor}
Let $\left(V,E,c,o,\Delta\left(=\Delta_{c}\right),\mathscr{H}_{E}\right)$
be as above. Let $c_{A}$ be a second conductance function defined
on $E$ and satisfying $c_{A}\geq c$ point-wise on $E$. (Note $c_{A}$
may be unbounded, even though $c$ might be bounded.) Let the operators
$j_{A}$ and $j_{A}^{*}$ be as above, i.e., 
\begin{equation}
\xymatrix{\mathscr{H}_{A}\ar@/^{1pc}/[r]^{j_{A}} & \mathscr{H}_{E}\ar@/^{1pc}/[l]^{j_{A}^{*}}}
\label{eq:ff2-1}
\end{equation}
then 
\begin{equation}
\mbox{Ker}\left(j_{A}\right)=\left\{ 0\right\} .\label{eq:ff2-2}
\end{equation}
\end{cor}
\begin{proof}
For operators in Hilbert space, we denote by $\mbox{Ker}$, and $\mbox{Ran}$,
kernel and range, respectively. We have 
\begin{equation}
\mbox{Ker}\left(j_{A}\right)=\mbox{Ran}\left(j_{A}^{*}\right)^{\perp}\label{eq:ff2-3}
\end{equation}
with ``$\perp$'' on the RHS in (\ref{eq:ff2-3}) denoting ortho-complement.
But by Lemma \ref{lem:jdipole} (eq. (\ref{eq:jdp})), we note that
all dipole vectors $v_{xy}^{\left(A\right)}$, for $x$ and $y$ in
$V$, are in $\mbox{Ran}\left(j_{A}^{*}\right)$. Since their span
is dense in $\mathscr{H}_{A}$, by Lemma \ref{lem:eframe}, the desired
conclusion (\ref{eq:ff2-2}) follows, i.e., $\mbox{Ker}\left(j_{A}\right)=\left\{ 0\right\} $.\end{proof}
\begin{cor}
Let $c$ and $c_{A}$ be as above, $c\leq c_{A}$ assumed to hold
on $E$. Then the partial isometry $W:\mathscr{H}_{A}\hookrightarrow\mathscr{H}_{E}$,
in 
\begin{equation}
j_{A}=W\left(j_{A}^{*}j_{A}\right)^{\frac{1}{2}}=\left(j_{A}^{*}j_{A}\right)^{\frac{1}{2}}W\label{eq:ff2-4}
\end{equation}
is isometric, i.e., 
\begin{equation}
W^{*}W=I_{\mathscr{H}_{A}}.\label{eq:ff2-5}
\end{equation}
\end{cor}
\begin{proof}
This is immediate from the previous corollary, and an application
of the polar-decomposition theorem for bounded operators, see e.g.,
\cite{DS88b,RS75}.\end{proof}
\begin{cor}
Let $V,E,c$ and $c_{A}$ be as above, with $c_{A}\geq c$, and let
the operators $j_{A}=W\left(j_{A}^{*}j_{A}\right)^{\frac{1}{2}}$,
$\Delta_{F}\left(=\left(\Delta_{c}\right)_{F}\right)$ and $\Delta_{F}^{\left(A\right)}$
be as described, i.e., the respective Krein extensions; then 
\begin{equation}
W\Delta_{F}^{\left(A\right)}W^{*}=\Delta_{F},\mbox{ and}\label{eq:ff2-6}
\end{equation}
\begin{equation}
W\Delta_{F}^{\left(A\right)}=\Delta_{F}W.\label{eq:ff2-7}
\end{equation}
\end{cor}
\begin{proof}
Immediate from the previous two corollaries. 

Note in particular that (\ref{eq:ff2-6})$\Longleftrightarrow$(\ref{eq:ff2-7})
in view of the isometry conclusion (\ref{eq:ff2-5}) in the last corollary.\end{proof}
\begin{thm}
\label{thm:cFri}Fix $G=\left(V,E\right)$. Let $o$ be a base point
in $V$, and set $V':=V\backslash\left\{ o\right\} $. Let $c$ and
$c_{A}$ be the conductance functions s.t. $c\leq c_{A}$ holds point-wise
on $E$. Let $\Delta$, $\Delta^{\left(A\right)}$ be the graph-Laplacians,
$v_{x}:=v_{xo}$, $v_{x}^{\left(A\right)}:=v_{xo}^{\left(A\right)}$,
$x\in V'$, be the dipoles. Then, 
\begin{equation}
j_{A}\Delta_{Kre}^{\left(A\right)}j_{A}^{*}=\Delta_{Kre}.\label{eq:F1}
\end{equation}

\end{thm}
The key step of the proof are the lemmas below.
\begin{lem}
Let $\left(V,E\right)$, $c,c^{A}$ and $L,L_{A}$ be as before, assume
$c\leq c^{A}$ point-wise on $E$; and let $j_{A}:\mathscr{H}_{A}\longrightarrow\mathscr{H}$
be the natural inclusion. 
\begin{enumerate}[label=(\roman{enumi})]
\item Then 
\begin{equation}
L_{A}^{*}j_{A}^{*}\subseteq L^{*};\label{eq:F1-6}
\end{equation}
i.e., the following diagram commutes.
\[
\xymatrix{ & \mathscr{H}_{A}\ar[rd]^{L_{A}^{*}}\\
\mathscr{H}\ar[rr]_{L^{*}}\ar[ur]^{j_{A}^{*}} &  & l^{2}
}
\]

\item We have $L\supseteq j_{A}L_{A}$.
\end{enumerate}
\end{lem}
\begin{proof}
Let $u=\sum\xi_{x}v_{x}$, $\sum\xi_{x}=0$, be any vector in $\mathscr{D}_{E}$
(eq. (\ref{eq:D2})); see part (\ref{enu:domL}) of Theorem \ref{thm:Deltaf}.
Then,
\begin{eqnarray*}
L_{A}^{*}\left(j_{A}^{*}u\right) & = & L_{A}^{*}\left(\sum\xi_{x}j_{A}^{*}\left(v_{x}\right)\right)\\
 & \underset{\text{lem }\ref{lem:jdipole}}{=} & L_{A}^{*}\left(\sum\xi_{x}v_{x}^{\left(A\right)}\right)\\
 & = & \left(\xi_{x}\right)=L^{*}u
\end{eqnarray*}
and the assertion follows.

Proof of part (ii): Taking adjoint of eq. (\ref{eq:F1-6}).
\end{proof}

\begin{lem}
\label{lem:Hext}Every selfadjoint $H$ operator with dense domain
in a Hilbert space is maximal Hermitian.\end{lem}
\begin{proof}
Suppose $H$ is selfadjoint, and that $T$ is a Hermitian extension;
then $H\subset T$, and by taking adjoints, we get $T^{*}\subseteq H^{*}$.
Hence 
\[
T\subseteq T^{*}\subset H^{*}=H^{*}\subseteq T
\]
so $T=H$.
\end{proof}

\begin{proof}[Proof of Theorem \ref{thm:cFri} ]
 We apply Theorem \ref{thm:Deltaf} (also see \cite{JT14}) to both
$\left(\Delta,\mathscr{H}_{C}\right)$ and $\left(\Delta^{\left(A\right)},\mathscr{H}_{A}\right)$,
and arrive at the following factorizations: 
\begin{align}
LL^{*} & =\Delta_{Kre}\;\mbox{in \ensuremath{\mathscr{H}}}\label{eq:F2}\\
L_{A}L_{A}^{*} & =\Delta_{Kre}^{\left(A\right)}\;\mbox{in }\mathscr{H}_{A}\label{eq:F3}
\end{align}
Then, by (\ref{eq:F2}), we have 
\begin{eqnarray*}
\Delta_{Kre} & \underset{\left(\ref{eq:F2}\right)}{=} & LL^{*}\\
 & \underset{\left(\ref{eq:F1-6}\right)}{\subseteq} & \left(j_{A}L_{A}\right)\left(L_{A}^{*}j_{A}^{*}\right)\\
 & = & j_{A}\left(L_{A}L_{A}^{*}\right)j_{A}^{*}\\
 & \underset{\left(\ref{eq:F3}\right)}{=} & j_{A}\Delta_{Kre}^{\left(A\right)}j_{A}^{*}.
\end{eqnarray*}

To finish the proof we apply Lemma \ref{lem:Hext} to the two operators
in $\mathscr{H}_{E}$, $H=\Delta_{Kre}$ and $T_{A}=j_{A}\Delta_{Kre}^{\left(A\right)}j_{A}^{*}$.
Here $H$ is selfadjoint on $\mathscr{H}_{E}$, and $T_{A}$ is a
Hermitian extension; hence by the Lemma \ref{lem:Hext}, $T_{A}=\Delta_{Kre}$. \end{proof}
\begin{rem}
Fix two conductance functions $c$ and $c_{A}$, and let $j_{A}$
be as before. See (\ref{eq:ff2-1}) and the diagram below: 
\[
\xymatrix{\mathscr{H}_{A}\ar@/^{1pc}/[r]^{\Delta_{F}^{\left(A\right)}}\ar@{_{(}->}[d]^{j_{A}} & \mathscr{H}_{A}\ar@{^{(}->}[d]_{j_{A}}\\
\mathscr{H}_{C}\ar@/_{1pc}/[r]_{\Delta_{F}}\ar@/^{1pc}/[u]^{j_{A}^{*}} & \mathscr{H}_{C}\ar@/_{1pc}/[u]_{j_{A}^{*}}
}
\]
Since for $u\in\mathscr{H}_{A}$, we have
\begin{equation}
\sum\sum c_{xy}\left|u\left(x\right)-u\left(y\right)\right|^{2}\leq\sum\sum c_{xy}^{\left(A\right)}\left|u\left(x\right)-u\left(y\right)\right|^{2}\label{eq:fc5}
\end{equation}
it follows that $j_{A}:\mathscr{H}_{A}\xrightarrow{j_{A}}\mathscr{H}_{C}$
is \emph{contractive}. 
\end{rem}
Since $j_{A},j_{A}^{*}$ are bounded operators, we have 
\begin{equation}
\left\Vert j_{A}^{*}\right\Vert _{\mathscr{H}_{A}}=\left\Vert j_{A}\right\Vert _{\mathscr{H}_{C}}\leq1.\label{eq:fc6}
\end{equation}

\begin{example}[Example \ref{ex:nn} revisited]
\label{exa:nn1}Let $V=\left\{ 0\right\} \cup\mathbb{Z}_{+}$, i.e.,
nearest neighbors. Let $c,c_{A}$ be two conductance functions given
by 
\begin{align*}
c_{n,n+1} & :=n\\
c_{n,n+1}^{\left(A\right)} & :=A^{n},\; A>1,\; n\in V.
\end{align*}
and we have 
\begin{align*}
\mathscr{H}_{C} & =\left\{ u:V\rightarrow\mathbb{C}\:\big|\:\sum_{n}n\left|u\left(n\right)-u\left(n+1\right)\right|^{2}<\infty\right\} \\
\mathscr{H}_{A} & =\left\{ u:V\rightarrow\mathbb{C}\:\big|\:\sum_{n}A^{n}\left|u\left(n\right)-u\left(n+1\right)\right|^{2}<\infty\right\} .
\end{align*}
Let $v_{n,n+1}$ and $v_{n,n+1}^{\left(A\right)}$ be the respective
dipoles, as illustrated in Fig \ref{fig:dipole}. 

Then, $j_{A}^{*}\left(v_{n,n+1}\right)=v_{n,n+1}^{\left(A\right)}$,
for all $n\in\mathbb{Z}_{+}$ (by (\ref{eq:jdp})), and 
\begin{align*}
\left\Vert v_{n,n+1}\right\Vert _{\mathscr{H}_{C}}^{2} & =1\\
\left\Vert v_{n,n+1}^{\left(A\right)}\right\Vert _{\mathscr{H}_{A}}^{2} & =\frac{1}{A^{n}}\rightarrow\infty\\
\left\Vert v_{n,n+1}\right\Vert _{\mathscr{H}_{A}}^{2} & =A^{n}\rightarrow0.
\end{align*}
\end{example}
\begin{lem}
In Example \ref{exa:nn1}, the operator $j_{A}j_{A}^{*}:\mathscr{H}_{C}\rightarrow\mathscr{H}_{C}$
is trace class, and 
\[
trace\left(j_{A}j_{A}^{*}\right)=\sum_{n=0}^{\infty}\frac{1}{A^{n}}.
\]
\end{lem}
\begin{proof}
Note the dipoles $\left\{ v_{n,n+1}:n\in V\right\} $ forms an ONB
in $\mathscr{H}_{C}$, and 
\begin{align*}
trace\left(j_{A}j_{A}^{*}\right) & =\sum_{n=0}^{\infty}\left\langle v_{n,n+1},j_{A}j_{A}^{*}v_{n,n+1}\right\rangle _{\mathscr{H}_{C}}\\
 & =\sum_{n=0}^{\infty}\left\langle j_{A}^{*}v_{n,n+1},j_{A}^{*}v_{n,n+1}\right\rangle _{\mathscr{H}_{A}}\\
 & =\sum_{n=0}^{\infty}\left\langle v_{n,n+1}^{\left(A\right)},v_{n,n+1}^{\left(A\right)}\right\rangle _{\mathscr{H}_{A}}\\
 & =\sum_{n=0}^{\infty}\frac{1}{A^{n}}<\infty.
\end{align*}

\end{proof}

\subsection{Comparing Spectra}

Let $\left(V,E,c,o,\Delta,\mathscr{H}_{E}\right)$ be a network, with
a fixed conductance function $c$ on $E$. We assume, as above, that
$\left(V,E,c\right)$ is connected, and that $\#V=\aleph_{0}$. 

We fix a second conductance function $c_{A}$ on $E$, and assume
the following estimate:
\begin{equation}
c_{A}\geq c\quad\mbox{holds on }E.\label{eq:cs1}
\end{equation}

Hence we get two energy Hilbert spaces $\mathscr{H}_{E}$ ($=\mathscr{H}_{C}$,
see (\ref{eq:Enorm-1})-(\ref{eq:Einner-1})) and $\mathscr{H}_{A}$
with the natural inclusion mapping $j_{A}:\mathscr{H}_{A}\hookrightarrow\mathscr{H}_{E}$
as follows
\begin{equation}
\xymatrix{\mathscr{H}_{A}\ar@/^{1pc}/[r]^{j_{A}} & \mathscr{H}_{E}\ar@/^{1pc}/[l]^{j_{A}*}}
\label{eq:cs2}
\end{equation}
and the respective graph Laplacians:
\begin{itemize}
\item $\Delta_{F}$ with its dense domain in $\mathscr{H}_{E}$; and
\item $\Delta_{F}^{\left(A\right)}$ with its dense domain in $\mathscr{H}_{A}$. 
\end{itemize}

We shall study the following notions from spectral theory of selfadjoint
operators:
\begin{defn}
\label{def:spmeas}Let $T$ be a selfadjoint operator with dense domain
in a Hilbert space $\mathscr{H}$; and let $P^{T}\left(\cdot\right)$
be the associate projection valued measure; see e.g., \cite{DS88b}.
For all vectors $u\in\mathscr{H}$, set 
\[
d\mu_{u}\left(\cdot\right)=\left\Vert P^{T}\left(\cdot\right)u\right\Vert ^{2}.
\]
$\mu_{u}$ is a finite Borel measure on $\mathbb{R}$, i.e., defined
on the Borel sigma algebra $\mathscr{B}$. We have:
\begin{equation}
u\in dom\left(T\right)\Longleftrightarrow\int_{\mathbb{R}}\lambda^{2}d\mu_{u}\left(\lambda\right)<\infty;\label{eq:cs3}
\end{equation}
and then 
\begin{equation}
\left\langle u,Tu\right\rangle =\int_{\mathbb{R}}\lambda d\mu_{u}\left(\lambda\right)\label{eq:cs4}
\end{equation}
where the integral is absolutely convergent. 

If $\psi$ is a Borel function, then a vector $u\in\mathscr{H}$ is
in the domain of the operator $\psi\left(T\right)$ iff $\psi\in L^{2}\left(\mu_{u}\right)$;
and then 
\begin{equation}
\left\Vert \psi\left(T\right)u\right\Vert ^{2}=\int\left|\psi\right|^{2}d\mu_{u}<\infty.\label{eq:cs5}
\end{equation}

\end{defn}

\begin{defn}
A finite positive Borel measure $\mu$ is said to be in $\mathfrak{M}\left(T\right)$
iff (Def.) $\exists u\in\mathscr{H}\backslash\left\{ 0\right\} $
such that 
\begin{equation}
d\mu=d\mu_{u}=d\left\Vert P^{T}\left(\cdot\right)u\right\Vert ^{2}.\label{eq:cs6}
\end{equation}
We say that $\mathfrak{M}\left(T\right)$ is the \emph{\uline{spectral
contents}} of the operator $T$.

If $\mu$ and $\nu$ are positive finite measures, we say that $\mu\ll\nu$
if the following implication holds:
\[
\left[S\in\mathscr{B},\;\nu\left(S\right)=0\right]\Longrightarrow\mu\left(S\right)=0;
\]
and we then introduce the Radon-Nikodym derivative $\frac{d\mu}{d\nu}$
i.e., $\frac{d\mu}{d\nu}\in L_{+}^{2}\left(\mathbb{R},\nu\right)$,
and 
\begin{equation}
\mu\left(S\right)=\int_{S}\frac{d\mu}{d\nu}\left(\lambda\right)d\nu\left(\lambda\right),\;\forall S\in\mathscr{B}.\label{eq:cs7}
\end{equation}
\end{defn}
\begin{thm}
\label{thm:Deltasp}Let $\mathscr{H}_{E}$ and $\mathscr{H}_{A}$
be as above, consider $u\in\mathscr{H}_{E}\backslash\left\{ 0\right\} $,
and $\mu=\mu_{u}\in\mathfrak{M}\left(\Delta_{Kr}\right)$; then $u_{A}:=j_{A}^{*}u$
yields $\mu_{u_{A}}^{\left(A\right)}\in\mathfrak{M}(\Delta_{Kr}^{\left(A\right)})$
if and only if $u_{A}\neq0$. In this case, we have 
\begin{equation}
\mu_{u}\ll\mu_{u_{A}}^{\left(A\right)}.\label{eq:cs8}
\end{equation}
\end{thm}
\begin{proof}
By Theorem \ref{thm:cFri} (eq. (\ref{eq:F1})) we have 
\[
j_{A}\Delta_{Kr}^{\left(A\right)}j_{A}^{*}=\Delta_{Kr},
\]
and $j_{A}\Delta_{Kr}^{\left(A\right)}j_{A}^{*}j_{A}\Delta_{Kr}^{\left(A\right)}j_{A}^{*}=\Delta_{Kr}^{2}$;
and so in the natural order of Hermitian operators, we have 
\[
\Delta_{Kr}^{2}\leq j_{A}\left(\Delta_{Kr}^{\left(A\right)}\right)^{2}j_{A}^{*}
\]
since $j_{A}j_{A}^{*}\leq I_{\mathscr{H}_{E}}$ by Lemma \ref{lem:jdipole}.
By induction, we get 
\begin{equation}
\Delta_{Kr}^{n}\leq j_{A}\left(\Delta_{Kr}^{\left(A\right)}\right)^{n}j_{A}^{*},\label{eq:cs9}
\end{equation}
for all $n\in\mathbb{Z}_{+}$, where again we use the order ``$\leq$''
on Hermitian operators in the Hilbert space $\mathscr{H}_{E}$. 

By Stone-Weierstra\textgreek{b} and measurable approximation, we get
the following estimate:
\begin{equation}
\left\Vert P_{\Delta_{Kr}}\left(S\right)u\right\Vert _{\mathscr{H}_{E}}^{2}\leq\left\Vert P_{\Delta_{Kr}^{\left(A\right)}}\left(S\right)j_{A}^{*}u\right\Vert _{\mathscr{H}_{A}}^{2}\label{eq:cs10}
\end{equation}
for all $S\in\mathscr{B}\left(\mathbb{R}\right)$ (= all Borel subsets
of $\mathbb{R}$). Now, introduce the measures from (\ref{eq:cs5})-(\ref{eq:cs6}),
and we get:
\begin{equation}
\mu_{u}^{\left(\Delta_{Kr}\right)}\left(S\right)\leq\mu_{j_{A}^{*}u}^{{\scriptscriptstyle (\Delta_{Kr}^{\left(A\right)})}}\left(S\right)\label{eq:cs11}
\end{equation}
for all $S\in\mathscr{B}\left(\mathbb{R}\right)$; which is the desired
conclusion in the theorem, see (\ref{eq:cs8}). From (\ref{eq:cs11}),
we are then able to verify relative absolute continuity for the respective
measures, and to compute the Radon-Nikodym derivatives: 
\[
d\mu_{u}^{\left(\Delta_{Kr}\right)}\Big/d\mu_{j_{A}^{*}u}^{({\scriptscriptstyle \Delta_{Kr}^{\left(A\right)}})}\in L_{+}^{1}\left(d\mu_{j_{A}^{*}u}^{({\scriptscriptstyle \Delta_{Kr}^{\left(A\right)}})}\right).
\]
\end{proof}
\begin{rem}
Eq (\ref{eq:cs10}) refers to extending information for a given intertwining
operator, intertwining a given pair of selfadjoint operators. We pass
from intertwining of the operators to intertwining of the respective
projection valued measures as follows. There are in four steps; first
from a given intertwining of the operators themselves, pass to the
intertwining property for polynomials in the respective selfadjoint
operators; then passing to the functional calculus for continuous
functions of compact support applied to the operators (Stone-Weierstrass),
and finally to the measurable functional calculus. When the latter
is applied to indicator functions for Borel subsets, we get the respective
projection valued measures corresponding to the given pair of selfadjoint
operators.
\end{rem}

\subsection{Comparing Harmonic Functions, and Dirac Masses}

Let $\left(V,E,c,\Delta,\mathscr{H}_{E}\right)$ be as before. Set
\begin{align*}
\mbox{Fin}\left(\mathscr{H}_{E}\right) & =\mathscr{H}_{E}\mbox{-closed span of }\left\{ \delta_{x}\:\big|\: x\in V\right\} \\
\mbox{Harm}\left(\mathscr{H}_{E}\right) & =\left\{ h\in\mathscr{H}_{E}\:\big|\:\Delta h=0\right\} 
\end{align*}

\begin{lem}
\begin{equation}
\mathscr{H}_{E}=\mbox{Fin}\left(\mathscr{H}_{E}\right)\oplus\mbox{Harm}\left(\mathscr{H}_{E}\right).\label{eq:harm1}
\end{equation}
\end{lem}
\begin{proof}
Immediate from the fact 
\begin{equation}
\left\langle \delta_{x},f\right\rangle _{\mathscr{H}_{E}}=\left(\Delta f\right)\left(x\right)=\sum_{y\sim x}c_{xy}\left(f\left(x\right)-f\left(y\right)\right),\label{eq:harm2}
\end{equation}
see Lemma \ref{lem:Delta} (\ref{enu:D5}).
\end{proof}
Let $c_{A}$ be a second conductance function, and assume $c_{A}\geq c$;
set $j_{A}:\mathscr{H}_{A}\rightarrow\mathscr{H}_{C}$, and $j_{A}^{*}:\mathscr{H}_{C}\rightarrow\mathscr{H}_{A}$
the adjoint, i.e., 
\[
\left\langle j_{A}^{*}\left(w\right),u\right\rangle _{\mathscr{H}_{A}}=\left\langle w,u\right\rangle _{\mathscr{H_{C}}}
\]
holds for all $u\in\mathscr{H}_{A}$. 
\begin{cor}
\begin{equation}
j_{A}^{*}\left(\mbox{Harm}\left(\mathscr{H}_{C}\right)\right)\subset\mbox{Harm}\left(\mathscr{H}_{A}\right).\label{eq:harm4}
\end{equation}
\end{cor}
\begin{proof}
Let $h\in\mbox{Harm}\left(\mathscr{H}_{C}\right)$, and let $x\in V$,
then 
\begin{align*}
\Delta_{A}\left(j_{A}^{*}\left(h\right)\right)\left(x\right) & =\left\langle j_{A}^{*}\left(h\right),\delta_{x}^{A}\right\rangle _{\mathscr{H}_{A}}\\
 & =\left\langle h,j_{A}\left(\delta_{x}^{A}\right)\right\rangle _{\mathscr{H}_{C}}\\
 & =\left\langle h,\delta_{x}\right\rangle _{\mathscr{H}_{C}}=\left(\Delta_{c}h\right)\left(x\right)=0;
\end{align*}
so $j_{A}^{*}\left(h\right)\in\mbox{Harm}\left(\mathscr{H}_{A}\right)$
as claimed.
\end{proof}
For the action of $j_{A}^{*}$ on the point masses $\delta_{x}$,
we then have the following:
\begin{thm}
\[
j_{A}^{*}\left(\delta_{x}\right)-\delta_{x}^{A}=\left(c\left(x\right)-c^{\left(A\right)}\left(x\right)\right)v_{x}^{\left(A\right)}-\sum_{y\sim x}\left(c_{xy}-c_{xy}^{\left(A\right)}\right)v_{xy}^{\left(A\right)},
\]
where $v_{x}$ and $v_{x}^{\left(A\right)}$ are the respective dipoles,
i.e., with a fixed base point $o\in V$, $v_{x}:=v_{xo}$, $v_{x}^{\left(A\right)}:=v_{xo}^{\left(A\right)}$,
for $x\in V':=V\backslash\left\{ o\right\} $. \end{thm}
\begin{proof}
We have $\delta_{x}=c\left(x\right)v_{x}-\sum_{y\sim x}c_{xy}v_{y}$
and therefore
\begin{eqnarray*}
j_{A}^{*}\left(\delta_{x}\right) & = & c\left(x\right)j_{A}^{*}\left(v_{x}\right)-\sum_{y\sim x}c_{xy}j_{A}^{*}\left(v_{y}\right)\\
 & = & c\left(x\right)v_{x}^{\left(A\right)}-\sum_{y\sim x}c_{xy}v_{y}^{\left(A\right)};
\end{eqnarray*}
and the desired formula now follows by a subtract, and 
\[
\delta_{x}^{A}=c^{\left(A\right)}\left(x\right)v_{x}^{\left(A\right)}-\sum_{y\sim x}c_{xy}^{\left(A\right)}v_{y}^{\left(A\right)}.
\]

\end{proof}
Below, we contrast our use of the Krein extensions (above) with the
corresponding Friedrichs extensions and their respective quadratic
forms.

Let $G=\left(V,E\right)$ be an infinite \emph{connected} network
as before. Fix two conductance functions $c,c^{\left(A\right)}:E\rightarrow\mathbb{R}_{+}$,
and assume $c_{xy}^{\left(A\right)}\geq c_{xy}$ for all $\left(xy\right)\in E$. 

\uline{Settings:}
\begin{itemize}[itemsep=1em]
\item $\mathscr{H}_{A},\mathscr{H}_{C}$ - the energy spaces with respect
to $c^{\left(A\right)}$ and $c$; i.e., the completion of span of
dipoles w.r.t. the corresponding energy inner products.
\item $j_{A}:\mathscr{H}_{A}\longrightarrow\mathscr{H}_{C}$ the natural
inclusion in Definition \ref{def:j}.
\item $v_{xy}$, $v_{xy}^{\left(A\right)}$ - dipoles in the respective
energy spaces.
\item $\delta_{x}$, $\delta_{x}^{\left(A\right)}$ - as defined in (\ref{eq:delx}).
\item $\Delta\left(=\Delta_{c}\right)$, $\Delta_{A}$ - graph Laplacians
in the respective energy spaces.
\end{itemize}
Recall that 
\[
j_{A}:\mathscr{H}_{A}\longrightarrow\mathscr{H}_{C}
\]
is a continuous injection, and 
\[
\left\Vert u\right\Vert _{C}\leq\left\Vert u\right\Vert _{A},\;\forall u\in\mathscr{H}_{A}.
\]
The operator $j_{A}j_{A}^{*}:\mathscr{H}_{C}\rightarrow\mathscr{H}_{C}$
is positive, selfadjoint, $\left\Vert j_{A}j_{A}^{*}\right\Vert _{C}\leq1$.

Let $\Delta_{A},\Delta$ be the graph-Laplacians, where 
\begin{align*}
dom(\Delta_{A}) & =span\{v_{xy}^{\left(A\right)}\}\\
dom(\Delta) & =span\{v_{xy}\}
\end{align*}
Set 
\[
M_{A}:=I_{\mathscr{H}_{A}}+\Delta_{A},\quad M:=I_{\mathscr{H}_{C}}+\Delta
\]
and let 
\begin{align*}
\mathscr{H}_{M_{A}} & =\mbox{completion of }dom(\Delta_{A})\mbox{ w.r.t }\left\Vert \varphi\right\Vert _{M_{A}}^{2}:=\left\langle \varphi,M_{A}\varphi\right\rangle _{A}\\
\mathscr{H}_{M} & =\mbox{completion of }dom(\Delta)\mbox{ w.r.t }\left\Vert \varphi\right\Vert _{M}^{2}:=\left\langle \varphi,M\varphi\right\rangle _{C}.
\end{align*}
By Lemma \ref{lem:FriedDomain}, and the assumption $c^{\left(A\right)}\geq c$,
we have 

\begin{align*}
\left\Vert u\right\Vert _{M_{A}} & \geq\left\Vert u\right\Vert _{M}\geq\left\Vert u\right\Vert _{C}\\
\left\Vert u\right\Vert _{M_{A}} & \geq\left\Vert u\right\Vert _{A}\geq\left\Vert u\right\Vert _{C},\;\forall u\in\mathscr{H}_{M_{A}}.
\end{align*}
Consequently, the inclusions in the diagram below are all contractive:
\[
\xymatrix{\mathscr{H}_{A}\ar[r]^{j_{A}} & \mathscr{H}_{C}\\
\mathscr{H}_{M_{A}}\ar[u]^{j_{1}}\ar[r] & \mathscr{H}_{M}\ar[u]_{j_{2}}
}
\]

Let $\widetilde{M}_{A}\supset M_{A}$, $\widetilde{M}\supset M$ be
the respective Krein extensions; i.e., 
\begin{align*}
\widetilde{M}_{A} & =I_{\mathscr{H}_{A}}+\Delta_{Kre}^{\left(A\right)}\\
\widetilde{M} & =I_{\mathscr{H}_{C}}+\Delta_{Kre}.
\end{align*}

\begin{lem}
The following hold.
\begin{enumerate}
\item \label{enu:Fc1}$dom(\widetilde{M})=j_{2}^{*}(\mathscr{H}_{C})$,
$\mathscr{H}_{M}=dom(\widetilde{M}^{1/2})$;
\item For all $u\in dom(\widetilde{M})$, $v\in\mathscr{H}_{M}$, 
\begin{equation}
\left\langle u,v\right\rangle _{M}=\langle\widetilde{M}u,v\rangle_{C};\label{eq:Fc}
\end{equation}

\end{enumerate}

and
\begin{enumerate}[resume]
\item \label{enu:Fa1}$dom(\widetilde{M}_{A})=j_{1}^{*}(\mathscr{H}_{A})$,
$\mathscr{H}_{M_{A}}=dom(\widetilde{M}_{A}^{1/2})$;\end{enumerate}
\begin{enumerate}
\item 
\begin{eqnarray*}
j_{A}\widetilde{M}_{A}j_{A}^{*} & = & j_{A}j_{A}^{*}+\Delta_{Fri}\\
 & = & \widetilde{M}+\left(j_{A}j_{A}^{*}-I_{\mathscr{H}_{C}}\right).
\end{eqnarray*}

\end{enumerate}
\end{lem}
\begin{proof}
See e.g., \cite{RS75,DS88b}.
\end{proof}

\begin{acknowledgement*}
The co-authors thank the following colleagues for helpful and enlightening
discussions: Professors Sergii Bezuglyi, Paul Muhly, Myung-Sin Song,
Wayne Polyzou, Gestur Olafsson, Keri Kornelson, and members in the
Math Physics seminar at the University of Iowa. We are grateful to
an anonymous referee for very helpful clarifications and improvements.
\end{acknowledgement*}
\bibliographystyle{amsalpha}
\bibliography{number8}

\end{document}